%% file: main.tex
\documentclass{article}
\usepackage[letterpaper,margin=1in]{geometry}

\newif\ifarxiv 
\arxivtrue
\date{\today} %

\usepackage[utf8]{inputenc} %
\usepackage[T1]{fontenc}    %
\usepackage{url}            %

\usepackage{lmodern}
\usepackage{microtype}      %

\usepackage{hhline}
\usepackage{multirow}

\usepackage[utf8]{inputenc} 
\usepackage[T1]{fontenc}

\usepackage{multicol}
\usepackage{silence}
\usepackage{amsfonts,thmtools,thm-restate}
\usepackage{mathtools}
\usepackage{xparse}
\usepackage{enumitem} 
\usepackage{etoolbox}
\usepackage{mathtools}
\usepackage{complexity}
\usepackage{svg}
\usepackage{xcolor, soul}
\usepackage{tablefootnote}
\usepackage[bb=boondox]{mathalpha} 
\definecolor{lightgrey}{rgb}{0.9,0.9,0.9}
\sethlcolor{lightgrey}

\usepackage[colorlinks=true, urlcolor=blue, linktoc=all]{hyperref}

\definecolor{mygray}{rgb}{0.6,0.6,0.6}

\usepackage[capitalise,nameinlink,noabbrev]{cleveref}  
\crefname{equation}{}{}

\newtheorem{theorem}{Theorem}[section]
\newtheorem{proposition}[theorem]{Proposition}
\newtheorem{lemma}[theorem]{Lemma}
\newtheorem{corollary}[theorem]{Corollary}
\newtheorem{definition}[theorem]{Definition}

\newtheorem{remark}[theorem]{Remark}

\hypersetup{
    colorlinks=true,
    linkcolor=blue,
    filecolor=magenta,
    urlcolor=blue,
    citecolor=blue,
    pdfborder={0 0 0}
}

\usepackage{tikz}
\usepackage{float}
\usepackage{wrapfig}
\usepackage{subcaption}

\usepackage{times}
\usepackage{xspace}

\input{config.tex}
\input{definitions.tex}
\input{algorithm_config.tex}
\input{bibliography_config.tex}

\renewcommand{\thefootnote}{\fnsymbol{footnote}}

\title{Complexity of Classical Acceleration for $\ell_1$-Regularized PageRank}

\author{%
  Kimon Fountoulakis\thanks{Equal contribution.}\\
  University of Waterloo, Canada\\
  \texttt{\href{mailto:kimon.fountoulakis@uwaterloo.ca}{kimon.fountoulakis@uwaterloo.ca}}
  \and
  David Martínez-Rubio\footnotemark[1]\\
  IMDEA Software Institute, Madrid, Spain\\
  \texttt{\href{mailto:david.martinezrubio@imdea.org}{david.martinezrubio@imdea.org}}
}

\begin{document}

\maketitle

\renewcommand{\thefootnote}{\arabic{footnote}}
\setcounter{footnote}{0}

\begin{abstract}
We study the degree-weighted work required to compute $\ell_1$-regularized PageRank using the standard accelerated proximal-gradient method (FISTA) \citep{beck2017first}.
For non-accelerated methods (ISTA) \citep{beck2017first}, the best known worst-case work is \inlinesmash{\widetilde{O}((\alpha\rho)^{-1})}, where $\alpha$ is the teleportation parameter and
$\rho$ is the $\ell_1$-regularization parameter.
It is not known whether classical acceleration methods can improve $1/\alpha$ to $1/\sqrt{\alpha}$ while preserving the $1/\rho$ locality scaling, or whether they can be asymptotically worse. For FISTA, we show a negative result by constructing a family of instances for which standard FISTA is asymptotically worse than ISTA. On the positive side, we analyze FISTA on a slightly over-regularized objective and show that, under a confinement condition, all spurious activations remain inside a boundary set $\mathcal{B}$.
This yields a bound consisting of an accelerated \inlinesmash{(\rho\sqrt{\alpha})^{-1}\log(\alpha/\varepsilon)} term plus a boundary overhead
\inlinesmash{\sqrt{\vol(\mathcal{B})}/(\rho\alpha^{3/2})}. We also provide graph-structural sufficient conditions that imply such confinement.
\end{abstract}

\section{Introduction}

Personalized PageRank (PPR) is a diffusion primitive that, from a seed node or distribution $s$, produces a nonnegative score vector concentrated near $s$, with applications to local graph clustering and ranking \citep{andersen2006local,gleich2015pagerank}. A key requirement is \emph{locality}: the running time to compute the vector should scale with the size of the target set of nodes, not the full graph. $\ell_1$ regularization is useful here because it induces sparsity. In the $\ell_1$-regularized PageRank formulation \citep{pmlr-v32-gleich14,fountoulakis2019variational}, one solves a strongly convex problem whose minimizer is sparse and nonnegative\footnote{A simple corollary of \cite{fountoulakis2019variational}: proximal-gradient iterates started at zero are nondecreasing.}. Concretely, for teleportation parameter $\alpha\in(0,1]$ and sparsity parameter $\rho>0$, we consider problems of the form
\[
\min_{x\in\R^n}
\;\underbrace{\frac12 x^\top Qx-\alpha\langle D^{-1/2}s,x\rangle}_{\text{smooth PageRank quadratic}}
\;+\;\underbrace{\alpha\rho\|D^{1/2}x\|_1}_{\text{$\ell_1$ sparsity penalty}},
\]
where $D$ is the degree matrix and $Q$ is a symmetric, scaled and shifted, Laplacian matrix, see
 \Cref{sec:preliminaries}.
Let $x^\star$ denote the unique minimizer %
and let $S^\star\defi\supp(x^\star)$ be its support. 

For the above problem, the primitives of first-order methods can be implemented locally: if an iterate is supported on a set $S$, evaluating its gradient and
performing a proximal gradient step only requires accessing edges incident to $S$. This motivates the degree-weighted work model \citep{fountoulakis2019variational}, in which  scanning the neighborhood of a vertex $i$ costs $d_i$ work, and the cost of a set $S$ of non-zero nodes is $\vol(S)\defi\sum_{i\in S} d_i$. The total work of an algorithm is the cumulative number of neighbor accesses with repetition performed over its execution.

\textbf{Motivation.} Accelerated first-order methods are worst-case optimal in gradient evaluations for smooth convex problems \citep{nesterov2004introductory,beck2009fast}. For $\ell_1$-regularized PageRank, however, the relevant measure is degree-weighted work, so the cost of an iteration depends on which coordinates are active. On undirected graphs, ISTA reaches a prescribed accuracy with worst-case total work \inlinesmash{\widetilde{O}((\alpha\rho)^{-1})} \citep{fountoulakis2019variational}. It is not known whether classical acceleration methods can improve $1/\alpha$ to $1/\sqrt{\alpha}$ while preserving the $1/\rho$ locality scaling, or whether they can be asymptotically worse. We study this question for the standard one-gradient-per-iteration FISTA method. The challenge is that extrapolation can create transient activations outside \(S^\star\), and even a few such activations can touch high-degree nodes and dominate the total work. We provide a negative worst-case result and a conditional upper bound on the total work.

\textbf{Worst-case negative result.}
We show that, on star graph instances with center degree $m$, ISTA remains supported on the seed leaf and therefore has graph-size-independent work. In contrast, standard FISTA  activates the high-degree center after two extrapolated steps, and incurs $\Omega(m)$ total degree-weighted work before reaching a fixed target accuracy. Thus, standard FISTA can be asymptotically worse than ISTA in the worst case.

\textbf{Total work bound and sufficient conditions.}
For FISTA run on a slightly over-regularized objective, under an explicit confinement condition ensuring that all spurious activations remain within a
boundary set $\mathcal{B}$, we obtain a work bound of the form
\[
\widetilde{O}\left(
\frac{1}{\rho\sqrt{\alpha}}\log\left(\frac{\alpha}{\varepsilon}\right)
\;+\;
\frac{\sqrt{\vol(\mathcal{B})}}{\rho\alpha^{3/2}}
\right).
\]
The first term is the accelerated cost of converging on the over-regularized problem;
the second term is an explicit overhead capturing the cumulative cost for exploring spurious nodes. We also give graph-structural sufficient conditions: a no-percolation criterion that makes the confinement hypothesis explicit once a candidate core set is specified. When this criterion holds for a set
\(S\) containing the relevant optimal support, it guarantees that momentum-induced activations cannot
percolate arbitrarily far into the graph: for all iterations \(k\), the iterates remain supported in
\(S \cup \bdry S\). In particular, any activation outside \(S\) is confined to the
vertex boundary \(\mathcal{B} \defi \bdry S\), so the locality overhead in our work bound is governed by the
boundary volume \(\vol(\mathcal{B})=\vol(\bdry S)\). This makes the second term interpretable as the cost of probing only the immediate neighborhood of the core region.

\textbf{Contributions.} Our main contributions can be summarized as follows.

\begin{itemize}[leftmargin=*]

    \item \textit{From KKT slack to cumulative spurious work, via over-regularization.}
    We show that activating an inactive coordinate forces a quantitative jump in per-iteration work controlled by its KKT slack, which together with FISTA's geometric contraction bounds cumulative spurious work. To avoid dependence on arbitrarily small slacks, we analyze a slightly over-regularized problem and use regularization-path monotonicity \citep{ha2021statistical} to absorb nearly active nodes into the true support, charging only clearly inactive ones.

    \item \textit{A conditional work bound for classic FISTA on a slightly over-regularized objective.}
    Under a boundary confinement condition (spurious activations stay within a boundary set $\mathcal{B}$),
    we obtain an explicit work bound with an accelerated term
    \inlinesmash{\widetilde{O}((\rho\sqrt{\alpha})^{-1}\log(\alpha/\varepsilon))} plus a boundary
    overhead quantified by $\sqrt{\vol(\mathcal{B})}/(\rho\alpha^{3/2})$ (cf.\ \Cref{thm:double_reg_work}).

    \item \textit{Graph-structural confinement guarantees and degree-based non-activation.}
    We give a sufficient no-percolation condition for boundary confinement, and in \cref{sec:fista_degree_nonactivation} we give a sufficient degree condition under which high-degree inactive nodes provably never activate under over-regularization.

    \item \textit{A negative worst-case result for standard FISTA.}
    We construct seed-at-leaf star instances for which standard FISTA activates a high-degree center and incurs $\Omega(m)$ total degree-weighted work to reach a fixed target accuracy, while ISTA remains supported on the seed and reaches the same target with
    \inlinesmash{O\!\left(\frac{1}{\alpha}\log\frac{1}{\varepsilon}\right)} work independent of $m$ (cf.\ \Cref{prop:lower_bound}).
    Thus standard FISTA can be asymptotically worse than ISTA in the degree-weighted work model.

\end{itemize}

\section{Related work}
\label{sec:related-work}

Personalized PageRank (PPR) is widely used for ranking and network analysis \citep{gleich2015pagerank}. A foundational locality result of \citet{andersen2006local} shows that an $\varepsilon$-approximate PPR vector can be computed in time $\tilde{O}(1/(\alpha\varepsilon))$ independent of graph size, enabling local graph partitioning.

\textbf{Variational formulations and worst-case locality for non-accelerated methods.} The variational perspective of \citet{pmlr-v32-gleich14,fountoulakis2019variational} shows that local clustering guarantees can be obtained by solving an $\ell_1$-regularized PageRank objective. \citet{fountoulakis2019variational} show that ISTA can be implemented locally with total work $\tilde{O}((\alpha\rho)^{-1})$, giving a worst-case graph-size-independent bound. An analogous result for standard accelerated methods, such as FISTA is an open problem. A related line studies statistical and path properties of these objectives; for instance, \citet{ha2021statistical} analyze the $\ell_1$-regularized PageRank solution path, which we leverage when reasoning about over-regularization.

\textbf{The COLT'22 open problem on acceleration and its solutions/attempts.} \citep{fountoulakis2022open} posed the COLT'22 open problem of whether one can obtain a provable accelerated algorithm for $\ell_1$-regularized PageRank with work $\widetilde{O}((\rho\sqrt{\alpha})^{-1})$, improving the $\alpha$-dependence by a factor of $1/\sqrt{\alpha}$ over ISTA while preserving locality. They emphasized that existing ISTA analyses do not cover acceleration and that it was unclear whether worst-case work might even degrade under acceleration. The first affirmative solution is due to \citet{martinezrubio2023accelerated}, who design accelerated
algorithms that retain sparse updates. Their method \emph{ASPR}
uses an expanding-subspace (outer-inner) scheme: it grows a set of ``good'' coordinates and runs an
accelerated projected gradient subroutine on the restricted feasible set. This yields a worst-case
bound of
$
\bigo{|S^\star|\widetilde{\operatorname{vol}}(S^\star)\alpha^{-1 / 2}\log(1 / \epsilon) + |S^\star|\vol(S^\star)}
$, where $S^\star$ is the support of the optimal solution and $\widetilde{\operatorname{vol}}(S^\star)$ is the number of edges of the subgraph formed only by nodes in $S^\star$. Compared to $\bigo{\vol(S^\star)\alpha^{-1}\log(1 / \epsilon)} = \bigotilde{1 / (\rho\alpha)}$ of ISTA, the solution improves the $\alpha$-dependence with a different sparsity dependence than ISTA.  In this work, we provide a support-sensitive, degree-weighted work analysis of the classic one-gradient-per-iteration FISTA method. Our contribution is algorithmically quite different, and the upper bound establishes a new trade-off under explicit confinement conditions on a candidate core set. \citet{zhou2024iterative} study locality for accelerated linear-system solvers and obtain an accelerated guarantee under an additional run-dependent residual-reduction assumption. In contrast, our bounds for standard FISTA are explicit and quantify when acceleration helps or hurts total work.

\textbf{Support identification, strict complementarity.} Our complementarity-gap viewpoint connects to the constraint-identification literature: under strict-complementarity-type conditions, proximal and proximal-gradient methods identify the optimal support in finitely many iterations \citep{burke1994exposing,nutini2018activeset,sun2019manifoldid}, and acceleration can delay identification via oscillations \citep{bareilles2020interplay} (see also \citep{wolfe1970convergence,guelat1986somecomments,garber2020revisiting}). These results are iteration-complexity statements under unit-cost steps and do not quantify locality-aware total work for accelerated methods.

\section{Preliminaries and notation}\label{sec:preliminaries}

We assume undirected and unweighted graphs, and
we use $[n] \defi \{1,\dots,n\}$. $\norm{\cdot}_2$ denotes the Euclidean norm and $\norm{\cdot}_1$ denotes the $\ell_1$ norm.
For a set $S\subseteq [n]$ we write $|S|$ for its cardinality.
If the indices in $S$ represent node indices of a graph, we use
$\vol(S)\defi\sum_{i\in S} d_i$ for the graph volume, where $d_i$ is the number of neighbors of node $i$, that is, its degree. We assume $d_i>0$ for all vertices.

We say a differentiable function $f$ is $L$-smooth if $\nabla f$ is $L$-Lipschitz with respect to $\norm{\cdot}_2$,
that is $\norm{\nabla f(x) - \nabla f(y)}_2 \leq L\norm{x - y}_2$.
We denote by $\mu > 0$ the strong-convexity parameter of a strongly-convex function $F$ with respect to $\norm{\cdot}_2$.
In such a case $F$ has a unique minimizer $x^\star$. 
For one such problem, define the optimal support and its complement as
$S^\star \defi \supp(x^\star)$ and $I^\star \defi [n]\setminus S^\star$.

The main objective that we consider in this work is the personalized PageRank quadratic objective, and its $\ell_1$ regularized version.
For a parameter $\alpha > 0$, called the teleportation parameter, and an initial distribution of nodes $s$
(i.e., $\innp{\ones, s} =1$, $s \geq 0$), the unregularized PageRank objective is
\[
    f(x) \defi \frac{1}{2} \innp{ x , Q x } - \alpha \innp{ D^{-1/2}s,  x}
    \ \text{ for } \ Q = \alpha I + \frac{1-\alpha}{2}\lapl,
\] 
where $\lapl \defi I - D^{-1/2}AD^{-1/2}$ is the symmetric normalized Laplacian matrix, which is known to satisfy
$0 \preccurlyeq \lapl \preccurlyeq  2 I$ \citep{butler2014spectral}. Thus, $\alpha I \preccurlyeq Q \preccurlyeq  I$,
which implies the objective is $\alpha$-strongly convex and $1$-smooth.
We will assume the seed is a single node $v$, that is $s=e_v$.
This is the case for clustering applications, where one seeks to find a cluster of nodes near $v$
that have high intraconnectivity and low connectivity to the rest of the graph
\citep{andersen2006local,gleich2015pagerank}.

A common objective for obtaining sparse PageRank solutions is the $\ell_1$-Regularized Personalized PageRank problem (RPPR),
which comes with the sparsity guarantee $\vol(S^\star) \leq 1 /{\rho}$, cf.\ \citep[Theorem 2]{fountoulakis2019variational},
where $\rho > 0$ is a regularization weight on the objective:
\begin{equation}\label{eq:reg_personalized_pagerank}
    \tag{RPPR}
    \min_{x \in \R^n} F_{\rho}(x),
    \qquad\text{where}\qquad
    F_{\rho}(x) \defi f(x) + g(x).
\end{equation}
where $g(x) \defi \alpha\rho\norm{D^{1 / 2}x}_1$. This is the central problem we study in this work.

It is worth noticing some properties of \cref{eq:reg_personalized_pagerank}.
The initial gap from $x_0=0$ is $\Delta_0 \defi F(0) - F(x^\star) \leq \alpha/2$, cf.\ \cref{lemma:initial_gap},
and so by strong convexity, the initial distance to $x^\star$ satisfies
$\norm{x^\star}_2 \leq \sqrt{2\Delta_0 / \mu} \leq 1$.
Finally, the minimizer $x^\star(\rho)$ of $F_\rho$ is coordinatewise nonnegative and the optimality conditions are, cf.\ \citep{fountoulakis2019variational}:
\begin{equation}\label{eq:KKT_coord_rho}
\nabla_i f\bigl(x^\star(\rho)\bigr)
\in
\begin{cases}
    \set{-\alpha\rho \sqrt{d_i}}, & x_i^\star(\rho)>0,\\
[-\alpha\rho \sqrt{d_i},0], & x_i^\star(\rho)=0.
\end{cases}
\end{equation}

\subsection{The FISTA Algorithm}

We introduce here the classical accelerated proximal-gradient method (FISTA) \citep{beck2009fast}
and the properties we use later.
We present the method for a composite objective $F(x)\defi f(x)+g(x)$ where $f$ is $L$-smooth and $F$
is $\mu$-strongly convex with respect to $\norm{\cdot}_2$.
For \cref{eq:reg_personalized_pagerank}, we have $L=1$ and $\mu=\alpha$ (since $\alpha I\preccurlyeq Q\preccurlyeq I$),
so the standard choice is step size $\eta=1/L=1$ and momentum parameter
$\beta = \frac{\sqrt{L/\mu} - 1}{\sqrt{L / \mu}+1}=\frac{1-\sqrt{\alpha}}{1+\sqrt{\alpha}}$. The iterates of the FISTA algorithm initialized with $x_{-1}=x_0=0$ are, for $k \geq 0$:
\begin{equation}\label{eq:apg_fista}\tag{FISTA}
y_k = x_k + \beta (x_k - x_{k-1}), \qquad
x_{k+1} = \prox_{\eta g}\!\bigl(y_k - \eta \nabla f(y_k)\bigr).
\end{equation}
The proximal operator is defined as
$\prox_{\eta g}(x) \defi \argmin_{y}\{\eta g(y) + \tfrac{1}{2}\|y-x\|_2^2\}$.
For the RPPR regularizer $g(x)=\alpha\rho\|D^{1/2}x\|_1$ the prox is separable and yields:
\begin{equation}\label{eq:prox_in_fista}
x_{k+1,i}
= \operatorname{sign}(y_{k,i}-\eta\nabla_i f(y_k))\,
\max\Bigl\{\,
\bigl|y_{k,i}-\eta\nabla_i f(y_k)\bigr|
-\eta\alpha\rho\sqrt{d_i},\,0
\Bigr\}.
\end{equation}

\begin{definition}
We measure runtime via a degree-weighted work model.
For an iterate pair $(y_k,x_{k+1})$ we define the per-iteration work as
\begin{equation}\label{eq:deg_work_model}
\mathrm{work}_k
\;\defi\;
\vol(\supp(y_k)) + \vol(\supp(x_{k+1})).
\end{equation}
For ISTA, $y_k=x_k$; for FISTA, $y_k=x_k+\beta(x_k-x_{k-1})$.
The total work to reach the stopping target is the sum of $\mathrm{work}_k$ over the iterations taken\footnote{Since each FISTA iteration computes a single gradient at $y_k$, one could alternatively take
$\mathrm{work}_k\defi \vol(\supp(y_k))$. Our definition \eqref{eq:deg_work_model} is a convenient
symmetric upper bound (it also covers evaluations at $x_{k+1}$, e.g., for stopping diagnostics), and
it matches the quantities controlled in our proofs up to an absolute constant.}.
\end{definition}

\section{FISTA's work analysis in RPPR}

We provide a lower bound and a conditional upper bound on the total work of \Cref{eq:apg_fista} on \Cref{eq:reg_personalized_pagerank}\footnote{All results in the paper have been formalized, subject to basic optimization results and results from previous papers. We provide details in \cref{sec:formalization}.}. First, the upper bound is proved by splitting the total work into a core cost and a spurious-exploration overhead. We run FISTA on the over-regularized objective $F_{2\rho}$, while taking as core set \(S=\supp(x^\star(\rho))\), so that \(\vol(S)\le 1/\rho\) by the RPPR sparsity guarantee (cf.\ \cref{sec:preliminaries}). The main task is then to bound the cumulative overhead from transient activations outside \(S\), using complementarity slacks, the confinement condition, and FISTA’s iteration complexity; this leads to the work bound proved in \cref{thm:double_reg_work}. We then complement this with a worst-case negative result showing that standard FISTA can be asymptotically worse than ISTA in \cref{sec:lower_bound}.

\subsection{Over-regularization}
A direct upper bound analysis of FISTA naturally runs into a margin issue.
In the arguments that follow, spurious activations will be controlled by KKT slacks at the optimum.
For RPPR, however, the smallest slack over inactive coordinates can be arbitrarily small
(see \cref{sec:bad_instances}), so any bound that depends on the minimum slack would be vacuous.
To obtain a work bound that remains meaningful, we will slightly over-regularize the objective\footnote{Our over-regularization affects clustering guarantees only by a constant factor, see \citep{andersen2006local,fountoulakis2019variational}.}, and we will relate the support of the solutions for $F_{2\rho}$  and $F_{\rho}$. For these two problems, we introduce the notation:
\[
g_A(x)\defi \alpha\rho\|D^{1/2}x\|_1
\qquad\text{and}\qquad
g_B(x)\defi 2\alpha\rho\|D^{1/2}x\|_1,
\]
and the corresponding minimizers
\[
x_A^\star \in \arg\min_x \bigl(f(x)+g_A(x)\bigr),
\qquad
x_B^\star \in \arg\min_x \bigl(f(x)+g_B(x)\bigr),
\]
with supports $S_A \defi \supp(x_A^\star),\; S_B \defi \supp(x_B^\star),\; I_B \defi [n]\setminus S_B$.
We run standard \cref{eq:apg_fista} on the over-regularized (B) problem, and we treat $S_A$ as a region where coordinates are potentially active at every iteration, even if some are inactive for  $x^\star_B$. This choice does not entail large work, since the guarantee is $\vol(S_A)\le 1/\rho$, cf. $\vol(S_B) \leq 1 / (2\rho)$.

We also have to account for the work of nodes that are active outside $S_A$. Define the spurious active set at step $k$ by \inlinesmash{
\widetilde{A}_k \defi \supp(x_{k+1})\cap S_A^c}.
Such activations are the only mechanism by which FISTA can
incur additional locality overhead beyond the cost of working inside $S_A$. Then, after $N$ iterations, the total degree-weighted work is bounded up to an absolute constant by
\begin{equation}\label{eq:running_time}
    \bigol{ N\vol(S_A) + \sum_{k=0}^{N-1} \vol(\widetilde{A}_k)}.
\end{equation}
The first term corresponds to the cost of running $N$ proximal-gradient steps while remaining in $S_A$,
since computing the gradient and applying the prox map costs work proportional to the volume of the active set.
The second term is the cumulative overhead from transient activations outside $S_A$.

Our goal is to bound \cref{eq:running_time}. The next subsection controls the second term. The complementarity-slack \Cref{lem:coord_jump} links spurious activations to deviations in the forward-gradient map, while \Cref{lem:two_tier_split} ensures a uniform slack bound outside $S_A$. Together, these remove dependence on tiny margins and allow summation of spurious volumes via FISTA's geometric contraction.

\subsection{Complementarity slack and spurious activations}

We formalize how momentum-induced activations outside the optimal support translate into a quantitative cost. For a coordinate that is zero at the optimum of the (B) problem, the KKT conditions define an interval for its gradient, and the distance to its boundary measures how safely inactive it is. If FISTA activates such a coordinate, the forward step must deviate by at least this margin, which allows us to bound the cumulative work on spurious supports.

Fix the (B) problem and let $x^\star \defi x_B^\star$.
For every inactive coordinate $i\in I_B$, define its degree-normalized complementarity (KKT) slack by
\begin{equation}\label{eq:gamma_i}
\gamma_i \defi \frac{\lambda_i-|\nabla f(x^\star)_i|}{\sqrt{d_i}}
= \frac{\lambda_i+\nabla_i f(x^\star)}{\sqrt{d_i}}
\ge 0,
\qquad \text{where } \lambda_i \defi 2\alpha\rho\sqrt{d_i}.
\end{equation}
The quantity $\gamma_i$ is the (degree-normalized) gap between the soft-threshold level $\lambda_i$ and the magnitude of the optimal gradient at coordinate $i$. In other words, it measures how far coordinate $i$ is from becoming active at the optimum. Define the gradient map and the set of spurious nodes by
\begin{equation*}\label{eq:new_u_def}
u(x)\defi x-\eta\nabla f(x),
\qquad
    A(y)\defi\supp\!\left(\prox_{\eta g_B}\bigl(u(y)\bigr)\right)\cap I_B \quad \text{ for any } x\in \R^n, y\in \R^n.
\end{equation*}
Here $u(\cdot)$ is the standard forward step used by proximal-gradient methods, and $A(y)$ is the subset of (B)-inactive indices that become nonzero after applying the prox map at the point $y$. The set $A(y_k)$ is exactly what creates the extra per-iteration cost due to spurious exploration with respect to an ideal local acceleration complexity of $\tilde{O}(1 / (\rho\sqrt{\alpha}))$.

We now formalize the connection between a spurious activation and a nontrivial deviation in the forward step. This is the basic bridge from optimality structure to the work bound.

\begin{lemma}\label{lem:coord_jump}\linktoproof{lem:coord_jump}
Fix $y\in\R^n$.
For every $i\in A(y)$, $
|u(y)_i-u(x^\star)_i| > \eta \gamma_i\sqrt{d_i}$.
\end{lemma}

\Cref{lem:coord_jump} is what allows us to turn spurious activations into a summable error budget that is compatible with FISTA's convergence rate, given that the distance to optimizer bounds $\norm{u(y_k)-u(x^\star)}$ and contracts with time.
Recall that the margin for the (B) problem can be written as
\begin{equation}\label{def:margin}
\gamma^{(B)}_i
\defi
2\rho\alpha + \frac{\nabla_i f(x_B^\star)}{\sqrt{d_i}}
\qquad \text{ for } i\in I_B.
\end{equation}
We now show that the margin of coordinates that are not in $S_A$ is large enough.

\begin{lemma}\label{lem:two_tier_split}\linktoproof{lem:two_tier_split}
Let $I_B\defi [n]\setminus S_B$ be the inactive set for problem (B), and define
\[
I_B^{\rm small}
\defi 
\Bigl\{ i\in I_B : \gamma^{(B)}_i < \rho\alpha \Bigr\},
\qquad
I_B^{\rm large}
\defi 
\Bigl\{ i\in I_B : \gamma^{(B)}_i \ge \rho\alpha \Bigr\}.
\]
Then $I_B^{\rm small}\subseteq S_A$.
\end{lemma}

Next, we compute a bound on the work $\vol(\widetilde{A}_k)$ which is proportional to the inverse minimum margin of the coordinates involved, and this quantity is no more than ${1}/{(\rho \alpha)}$ by the lemma above.

\subsection{Work bound and sufficient conditions}

We now derive a conditional upper bound on the work.
Recall the decomposition \eqref{eq:running_time}, the uniform bound for the margin of coordinates in \inlinesmash{\widetilde{A}_k} in the previous section, together with Cauchy-Schwarz and the distance contraction of $\|y_k-x_B^\star\|_2$ that we show in \Cref{lem:yk_distance}, makes the series \inlinesmash{\sum_k \vol(\widetilde A_k)} summable and leads to the overhead term in the theorem below.

\begin{theorem}\label{thm:double_reg_work}\linktoproof{thm:double_reg_work}
For the (B) problem with objective $F_B(x)=f(x)+g_B(x)$, run \cref{eq:apg_fista}.
Let $\mathcal{B}$ be a set such that \inlinesmash{\widetilde{A}_k \subseteq \mathcal{B}} for all $k\ge 0$, Then, we reach $F_B(x_N)-F_B(x_B^\star)\le \varepsilon$ after a total degree-weighted work of at most
\[
\mathrm{Work}(N_\varepsilon)
\;\le\;
O\left(
    \frac{1}{\rho\sqrt{\alpha}}\log\left(\frac{\alpha}{\varepsilon}\right)
\;+\;
    \frac{\sqrt{\vol(\mathcal{B})}}{\rho\alpha^{3 / 2}}
\right).
\]
\end{theorem}

The bound in \Cref{thm:double_reg_work} separates the cost of converging on $S_A$, from the extra cost of transient exploration.
The first term is the baseline accelerated contribution: FISTA needs $N_\varepsilon=O\!\left((\sqrt{\alpha})^{-1}\log(\alpha/\varepsilon)\right)$ iterations, and each iteration costs at most a constant times $\vol(S_A)\le 1/\rho$, yielding $\bigo{(\rho\sqrt{\alpha})^{-1}\log(\alpha/\varepsilon)}$.
The second term bounds the entire cumulative volume of the spurious sets \inlinesmash{\widetilde A_k}.
The factor $\rho^{-1}\alpha^{-3/2}$ reflects the combination of (i) the uniform margin $\rho\alpha$ obtained from \Cref{lem:two_tier_split} and (ii) the geometric contraction of the iterates, which sums to $O(\alpha^{-1/2})$.

We used hypothesis $\widetilde{A}_k\subseteq\mathcal{B}$ as an explicit locality requirement. We now give a graph-structural sufficient condition that implies such confinement, with $\mathcal{B}$ identified as a vertex boundary.

\begin{theorem}\label{thm:boundary_confinement_exposure}\linktoproof{thm:boundary_confinement_exposure}
    Consider the (B) problem objective $F_{2\rho}$ in \cref{eq:reg_personalized_pagerank}, with $\alpha\in(0,1)$, and let $S$ be a set such that $S_B \subseteq S$.
Define $\bdry S$ as the vertex boundary of $S$ and $\mathrm{Ext}(S) \defi V\setminus(S\cup \bdry S)$.
Assume that for all $i\in \mathrm{Ext}(S)$,
\begin{equation}\label{eq:exposure_assumption}
    \frac{|\N(\{i\})\cap \bdry S|}{d_i}\le \left(\frac{\alpha\rho}{2(1-\alpha)}\right)^2 d_i d_{\min \bdry S},
\end{equation}
where $d_{\min \bdry S} \defi \min_{j\in \bdry S} d_j$.
Then the iterates of \cref{eq:apg_fista} satisfy $\supp(x_k)\subseteq S\cup \bdry S$ for all $k\ge 0$:
\[
\supp(x_k)\subseteq S\cup \bdry S.
\]
\end{theorem}

\cref{thm:boundary_confinement_exposure} gives a mechanism for ruling out spurious activations outside a candidate region. Condition~\eqref{eq:exposure_assumption} upper-bounds the fraction of an exterior node's incident edges that connect to the boundary $\partial S$, preventing extrapolated FISTA iterates from percolating into the exterior.

\begin{remark}
If $S=S_A$ in \cref{thm:boundary_confinement_exposure},  then $\supp(x_k)\subseteq S_A\cup \bdry S_A$ for all $k\ge 0$. So any spurious activation happens in $\mathcal{B}\defi \bdry S_A$, that is, the confinement hypothesis in \Cref{thm:double_reg_work} holds with $\mathcal{B}= \bdry S_A$, and the overhead term in the work bound is governed by the boundary volume $\vol(\bdry S_A)$.
\end{remark}

\begin{remark}
In \cref{sec:fista_degree_nonactivation}, we prove a complementary property showing that nodes of high-enough degree do not ever get activated, which implies that the $\vol(\mathcal{B})$ term in \cref{thm:double_reg_work} can be reduced to the volume of nodes in $\mathcal{B}$ with degree below the threshold given there.
\end{remark}

\subsection{FISTA can be worse than ISTA}
\label{subsec:lower_bound}

The lower bound comes from a seed-at-leaf star instance: the optimum is supported only on the seed leaf, so ISTA stays local, but FISTA activates the high-degree center after two extrapolated steps. Since the center has degree \(m\), this creates \(\Omega(m)\) degree-weighted work before the method can reach the target accuracy. The full construction and proof are deferred to \Cref{sec:lower_bound,prop:lower_bound}.

\begin{proposition}[Informal]
Fix \(\alpha\in(0,1)\). There exists a seed-at-leaf star instance whose center has degree \(m\), and a threshold \(\varepsilon_0(\alpha)>0\), such that standard FISTA needs at least \(2m\) total work to reach any accuracy \(0<\varepsilon\le \varepsilon_0(\alpha)\). In contrast, ISTA needs \(O\!\left(\frac{1}{\alpha}\log\frac{1}{\varepsilon}\right)\) work, independent of \(m\).
\end{proposition}

\section{Experiments}\label{sec:experiments}

This section evaluates when FISTA reduces (and when it can increase) the total work for $\ell_1$-regularized PageRank, reflecting the tradeoff in \Cref{thm:double_reg_work}.
We present two sets of experiments.
First, we consider a controlled synthetic core-boundary-exterior graph family. Details on parameter tuning are given in \cref{app:exp_setting_details}\footnote{In our experiments, we did not over-regularize the problem.}.
Second, we compare ISTA and FISTA on real data: SNAP \citep{leskovec2014snap} graphs. 
\begin{figure}[t!]
    \centering
    \includegraphics[width=0.9\linewidth]{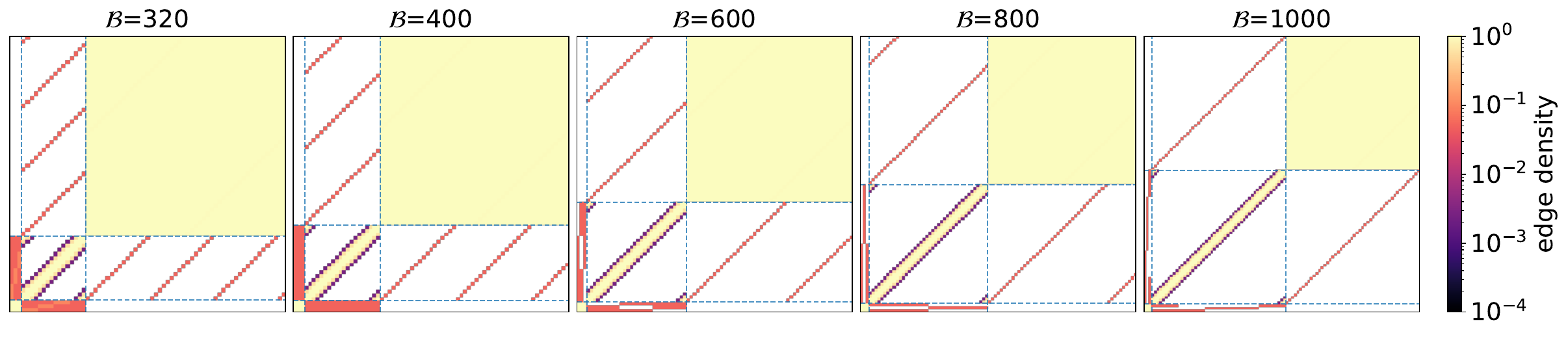}
    \caption{\textit{Adjacency density.}
    For each boundary size $|\mathcal{B}|$ we visualize the adjacency matrix via a
    binned edge-density heatmap (bin size $20$), where each pixel shows the fraction of possible
    edges between a pair of bins (log-scaled; colormap magma with white below $10^{-4}$).
    Dashed lines mark the core | boundary | exterior block boundaries.
    The plots show the clique (upper-left block), the boundary circulant band,
    the nearly dense exterior block, and the sparse cross-region interfaces.
    }
    \label{fig:adjacency_binned_density}
\end{figure}

For synthetic experiments the no-percolation assumption is satisfied. We use a three-block node partition
$V = S \cup \mathcal{B} \cup \mathrm{Ext}$, where $S$ (the core) contains the seed.
The induced subgraph on $S$ is a clique, while $\mathcal{B}$ (the boundary) and
$\mathrm{Ext}$ (the exterior) are each internally connected. Cross-region connectivity is sparse:
the core connects to $\mathcal{B}$ with a fixed per-core fan-out, and each exterior node has
one neighbor in $\mathcal{B}$. This yields a block structure in the adjacency matrix and lets us
vary the boundary size/volume $\vol(\mathcal{B})$ while keeping the core fixed, see \Cref{fig:adjacency_binned_density}. We provide details in \cref{app:exp_setting_details}
\ifarxiv
\footnote{Code that reproduces all experiments is available at \url{https://github.com/watcl-lab/accelerated_l1_PageRank_experiments}.}. 
\else
\footnote{Code that reproduces all experiments is available in the supplementary material.}
\fi

\subsection{Synthetic boundary-volume sweep experiment}
\label{subsec:synthetic_volB_sweep}

\begin{wrapfigure}{r}{0.35\linewidth}
    \centering
    \includegraphics[width=\linewidth]{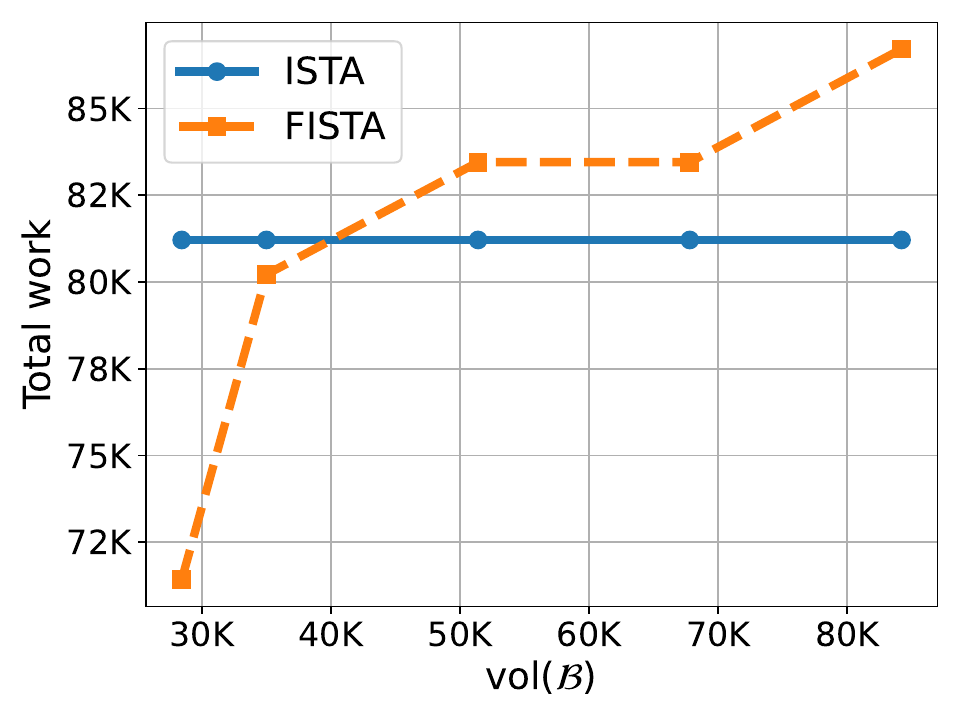}
    \caption{\textit{Work vs.\ $\vol(\mathcal{B})$.}
    Work by ISTA and FISTA against $\vol(\mathcal{B})$.}
    \label{fig:work_vs_volumeB}
\end{wrapfigure}

This section provides synthetic experiments illustrating how the boundary volume can dominate the running time of accelerated proximal-gradient methods. In particular, the experiment isolates the mechanism behind the
$\sqrt{\vol(\mathcal{B})}$-term in the work bound in \Cref{thm:double_reg_work},
and shows empirically that, as $\vol(\mathcal{B})$ increases, the accelerated method can become slower than its non-accelerated counterpart.

We use the core-boundary-exterior synthetic construction from \Cref{sec:experiments}, and vary only the boundary size $|\mathcal{B}|$ (and hence $\vol(\mathcal{B})$), keeping the core, exterior, and all degree/connectivity parameters fixed. On each instance of the sweep we solve the $\ell_1$-regularized PageRank objective \eqref{eq:reg_personalized_pagerank}
with $\alpha=0.20$ and $\rho=10^{-4}$, comparing ISTA and FISTA under the common initialization,
parameter choices, stopping protocol, and work accounting described in
\Cref{sec:experiments}. We set $\varepsilon=10^{-6}$. 

\Cref{fig:work_vs_volumeB} plots the work to reach the common stopping target
as a function of $\vol(\mathcal{B})$.
The key trend is that FISTA becomes increasingly expensive as $\vol(\mathcal{B})$ grows. For sufficiently large $\vol(\mathcal{B})$ it becomes slower than ISTA.
This is exactly the behavior suggested by the bound in \Cref{thm:double_reg_work}
(and in particular by the $\sqrt{\vol(\mathcal{B})}/(\rho\alpha^{3/2})$ term):
as the boundary volume grows, the potential cost of spurious exploration in the boundary grows as well.

\subsection{Sweeps in \texorpdfstring{$\rho$}{rho}, \texorpdfstring{$\alpha$}{alpha} and \texorpdfstring{$\varepsilon$}{epsilon} at fixed boundary size}
\label{subsec:b600_sweeps}

We fix $|\mathcal{B}|=600$, and
we run three sweeps (summarized in \Cref{fig:B600_sweeps_1x4}): (i) an $\rho$-sweep, reported for both a dense-core (clique) instance and a sparse-core variant (connected, $20\%$ of clique
edges) to confirm that the observed $\rho$-dependence is not an artifact of the symmetric clique core; (ii) an $\alpha$-sweep
with $\rho=10^{-4}$; and (iii) an $\varepsilon$-sweep at fixed $\alpha=0.20$ (complete details in
\cref{app:b600_sweeps_full}).

The $\rho$ sweeps (\cref{fig:B600_work_vs_rho_dense,fig:B600_work_vs_rho_sparse}) show that work decreases as $\rho$ increases and collapses to $0$ beyond the
trivial-solution threshold; across the sweep, ISTA and FISTA exhibit qualitatively similar $1/\rho$-type scaling,
consistent with the $\rho$-dependence in \Cref{thm:double_reg_work} for fixed $\alpha$ and fixed boundary size.
The $\alpha$ sweep (\cref{fig:B600_work_vs_alpha}) shows increasing work as $\alpha$ decreases, and FISTA can be
slower than ISTA over a substantial small-$\alpha$ range, consistent with the interpretation of \Cref{thm:double_reg_work}.
Finally, the $\varepsilon$ sweep (\cref{fig:B600_work_vs_epsilon}) shows increasing work as the tolerance decreases. FISTA is faster for small $\varepsilon$, consistent with the interpretation of \Cref{thm:double_reg_work}.

\begin{figure}[htb!]
    \centering
    \begin{subfigure}[t]{0.25\linewidth}
        \centering
        \includegraphics[width=\linewidth]{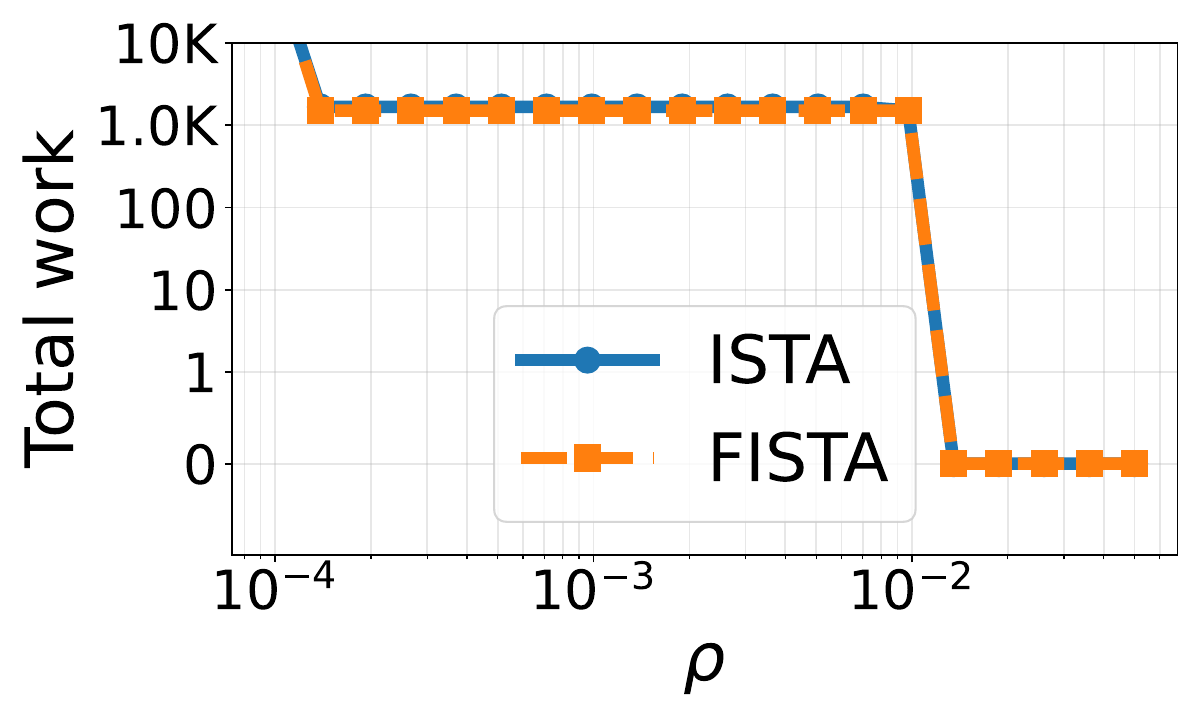}
        \caption{\textit{Dense core (clique).}}
        \label{fig:B600_work_vs_rho_dense}
    \end{subfigure}\hfill
    \begin{subfigure}[t]{0.25\linewidth}
        \centering
        \includegraphics[width=\linewidth]{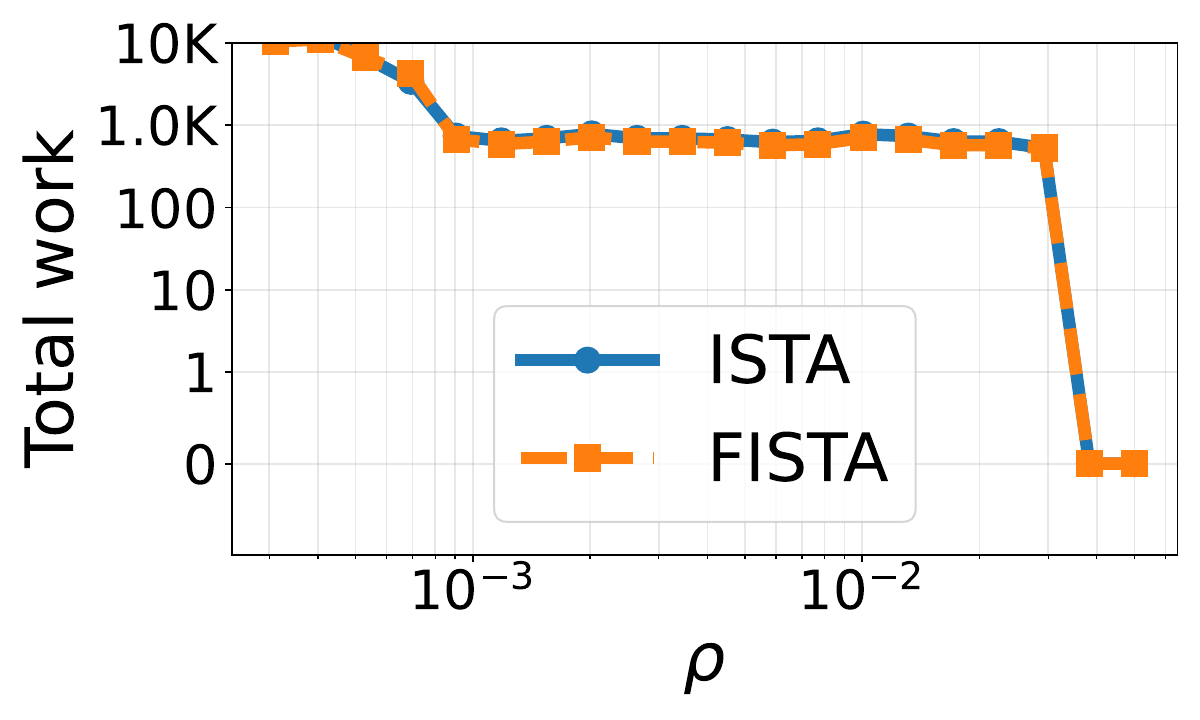}
        \caption{\textit{Sparse core.}}
        \label{fig:B600_work_vs_rho_sparse}
    \end{subfigure}\hfill
    \begin{subfigure}[t]{0.25\linewidth}
        \centering
        \includegraphics[width=\linewidth]{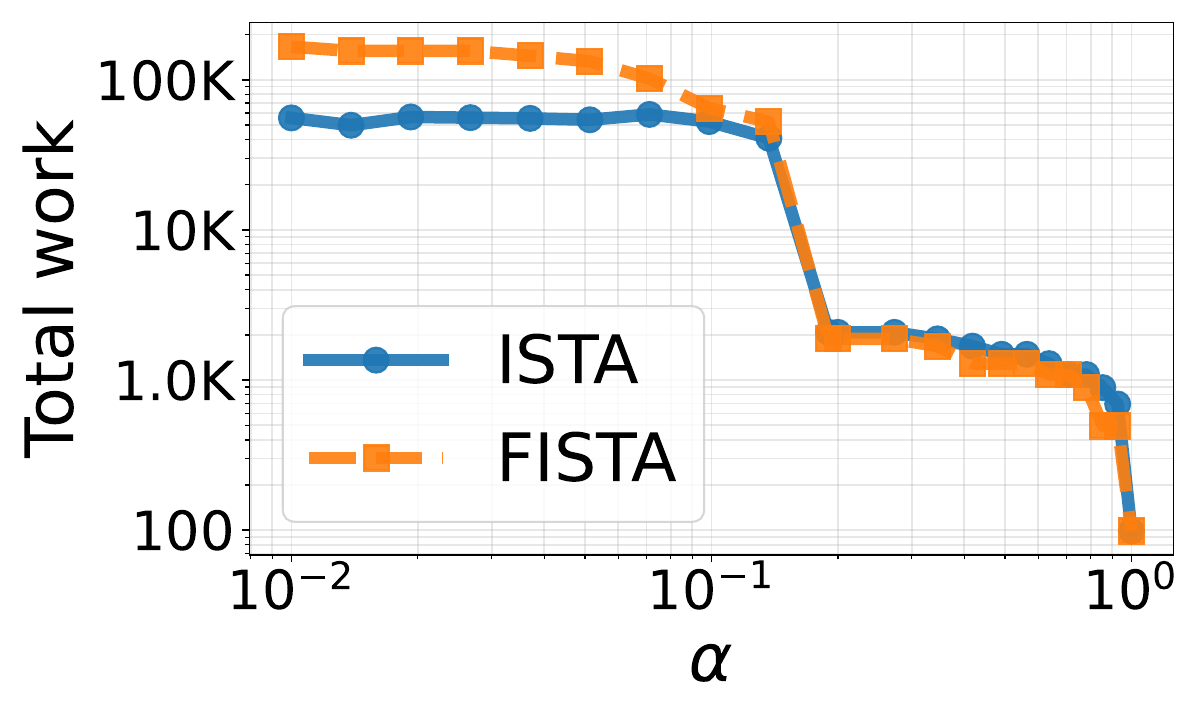}
        \caption{\textit{$\alpha$ sweep}}
        \label{fig:B600_work_vs_alpha}
    \end{subfigure}\hfill
    \begin{subfigure}[t]{0.25\linewidth}
        \centering
        \includegraphics[width=\linewidth]{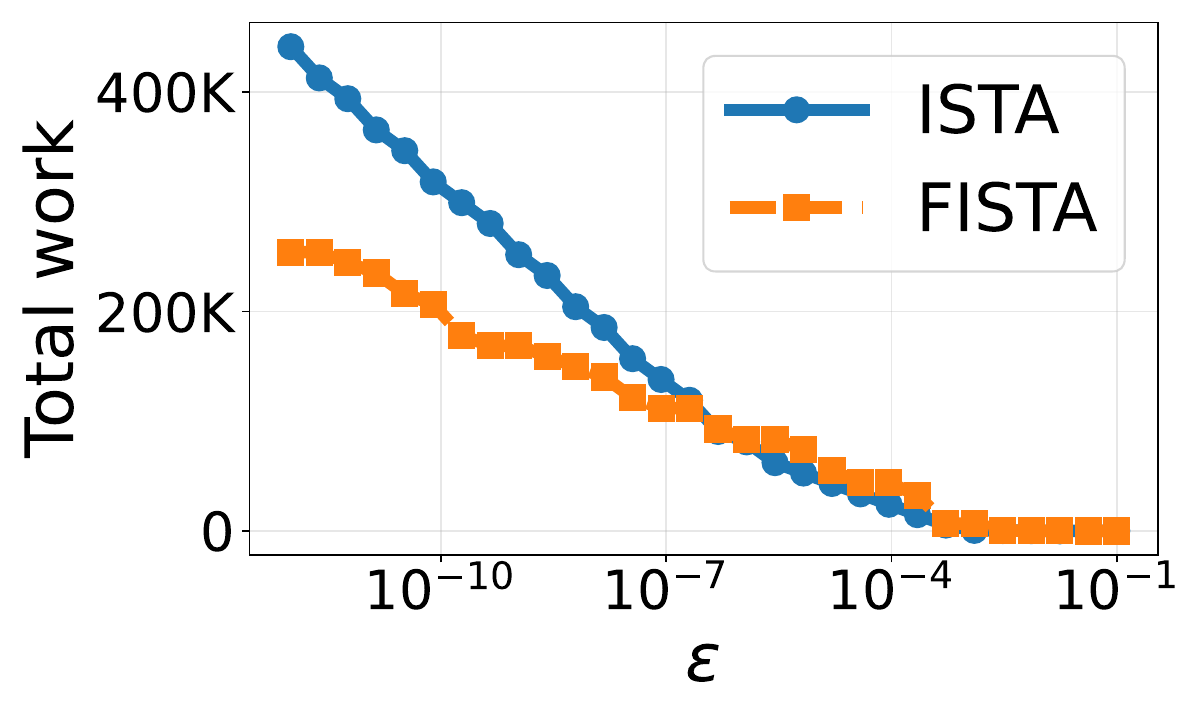}
        \caption{\textit{$\varepsilon$ sweep}}
        \label{fig:B600_work_vs_epsilon}
    \end{subfigure}
    \caption{\textit{Sweeps at fixed $|\mathcal{B}|=600$.}
    \cref{fig:B600_work_vs_rho_dense} shows the $\rho$-sweep with a dense core (clique) on a fresh randomized graph per $\rho$; \cref{fig:B600_work_vs_rho_sparse} shows the $\rho$-sweep with a sparse core (connected, $20\%$ of clique edges) on a fresh randomized graph per $\rho$.
    \cref{fig:B600_work_vs_alpha} sweeps $\alpha$ at a fixed tolerance $\varepsilon=10^{-6}$ on a single instance constructed so that the no-percolation condition holds at the smallest swept value, with parameters selected by an inexpensive auto-tuning step (\cref{app:b600_sweeps_full}).
    \cref{fig:B600_work_vs_epsilon} sweeps the tolerance $\varepsilon$ at fixed $\alpha=0.20$ on the baseline unweighted instance.}
    \label{fig:B600_sweeps_1x4}
\end{figure}

\subsection{Real-data benchmarks on SNAP graphs}
\label{subsec:real_data_experiments}

Our synthetic experiments use a deliberate core-boundary-exterior construction in order to satisfy the no-percolation assumption.
The real-data benchmarks in this subsection are at the opposite end of the spectrum: heterogeneous SNAP \citep{leskovec2014snap}
networks whose connectivity and degree profiles are not engineered to fit the synthetic template.
We include these datasets to illustrate both a positive and a negative real example for acceleration. On \texttt{com-Amazon}, \texttt{com-DBLP}, and \texttt{com-Youtube} we
typically observe a consistent work reduction with FISTA, whereas \texttt{com-Orkut} exhibits a setting where FISTA can be slower due to costly exploration beyond the optimal support. \footnote{For each dataset we sample $300$ seed nodes uniformly at random from the
non-isolated vertices; the same seed set is reused for both ISTA and FISTA and across all sweeps. In the sweep plots below, solid lines are means over the $300$ seeds and the shaded bands show the interquartile range (25\%-75\%).}.

\textbf{Work vs.\ parameter $\alpha$.}
We sweep $\alpha$ over a log-spaced grid in $[10^{-3},0.9]$ while fixing $\rho=10^{-4}$ and
$\varepsilon=10^{-8}$\footnote{Note that $\varepsilon = 10^{-8}$ is different from the value used in the synthetic experiments, which was $\varepsilon = 10^{-6}$. This is because we observed that the latter setting was too large for the real data to produce meaningful plots.}; results are in \Cref{fig:real_work_vs_alpha}. On \texttt{com-Amazon}, \texttt{com-DBLP}, and \texttt{com-Youtube}, FISTA consistently reduces work relative to ISTA
across the full $\alpha$ range.
On \texttt{com-Orkut}, however, FISTA can be slower than ISTA for small $\alpha$ before becoming competitive again
at moderate and large $\alpha$, illustrating that acceleration can lose under our work metric.
\begin{figure}[t!]
    \centering
    \begin{subfigure}[t]{0.25\linewidth}
        \centering
        \includegraphics[width=\linewidth]{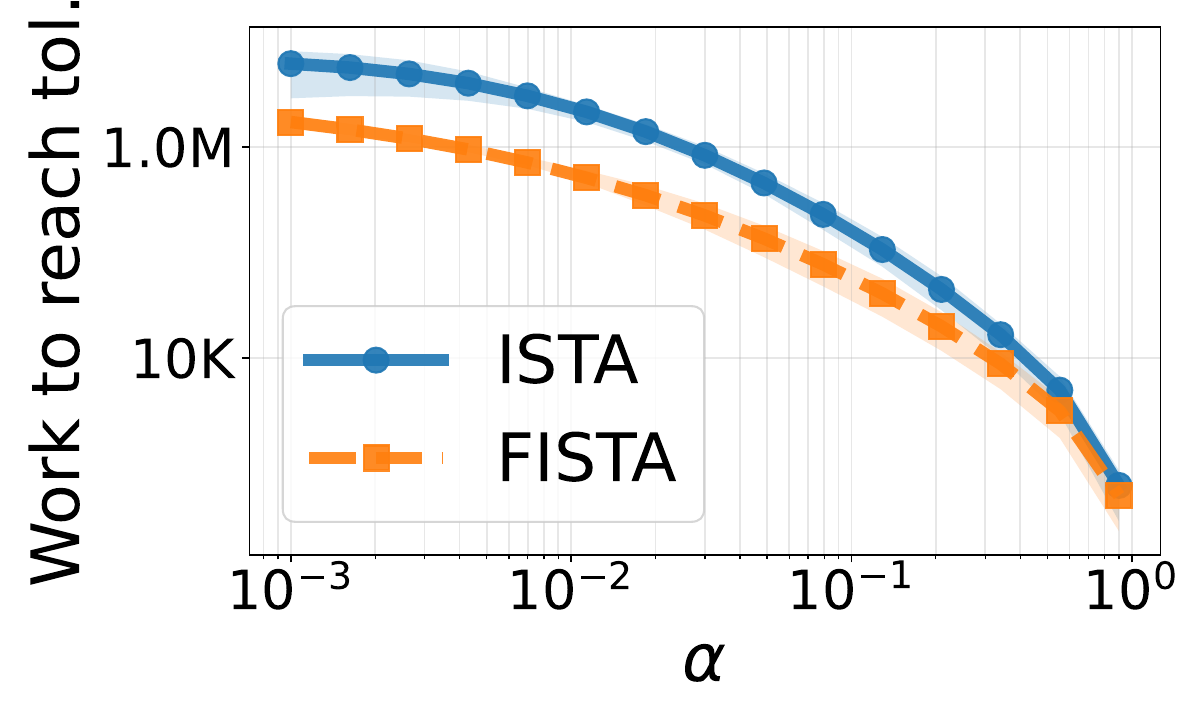}
        \caption{\texttt{com-Amazon}.}
        \label{fig:real_work_vs_alpha_amazon}
    \end{subfigure}\hfill
    \begin{subfigure}[t]{0.25\linewidth}
        \centering
        \includegraphics[width=\linewidth]{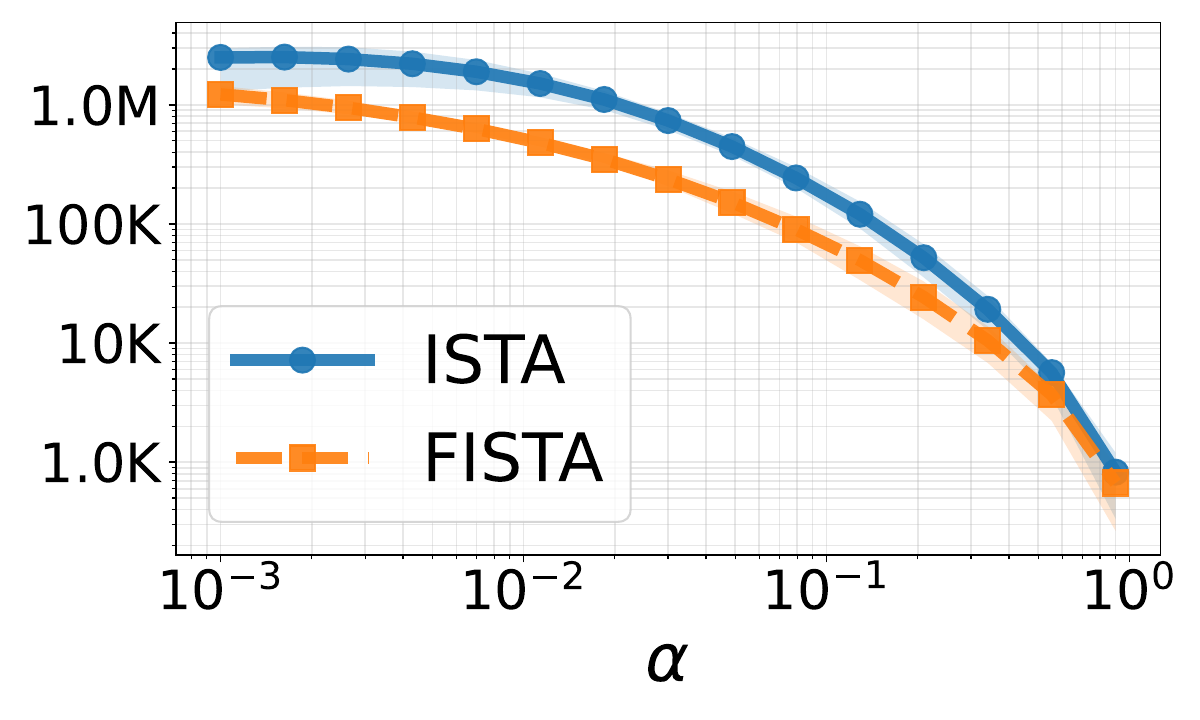}
        \caption{\texttt{com-DBLP}.}
        \label{fig:real_work_vs_alpha_dblp}
    \end{subfigure}\hfill
    \begin{subfigure}[t]{0.25\linewidth}
        \centering
        \includegraphics[width=\linewidth]{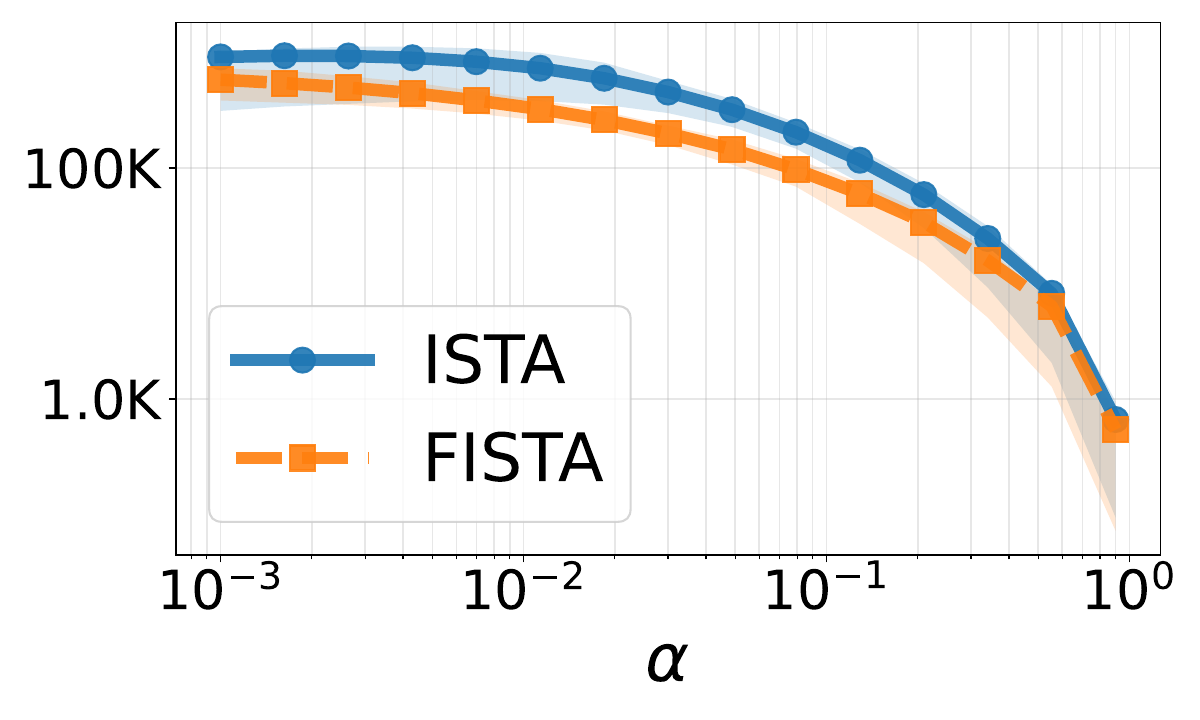}
        \caption{\texttt{com-Youtube}.}
        \label{fig:real_work_vs_alpha_youtube}
    \end{subfigure}\hfill
    \begin{subfigure}[t]{0.25\linewidth}
        \centering
        \includegraphics[width=\linewidth]{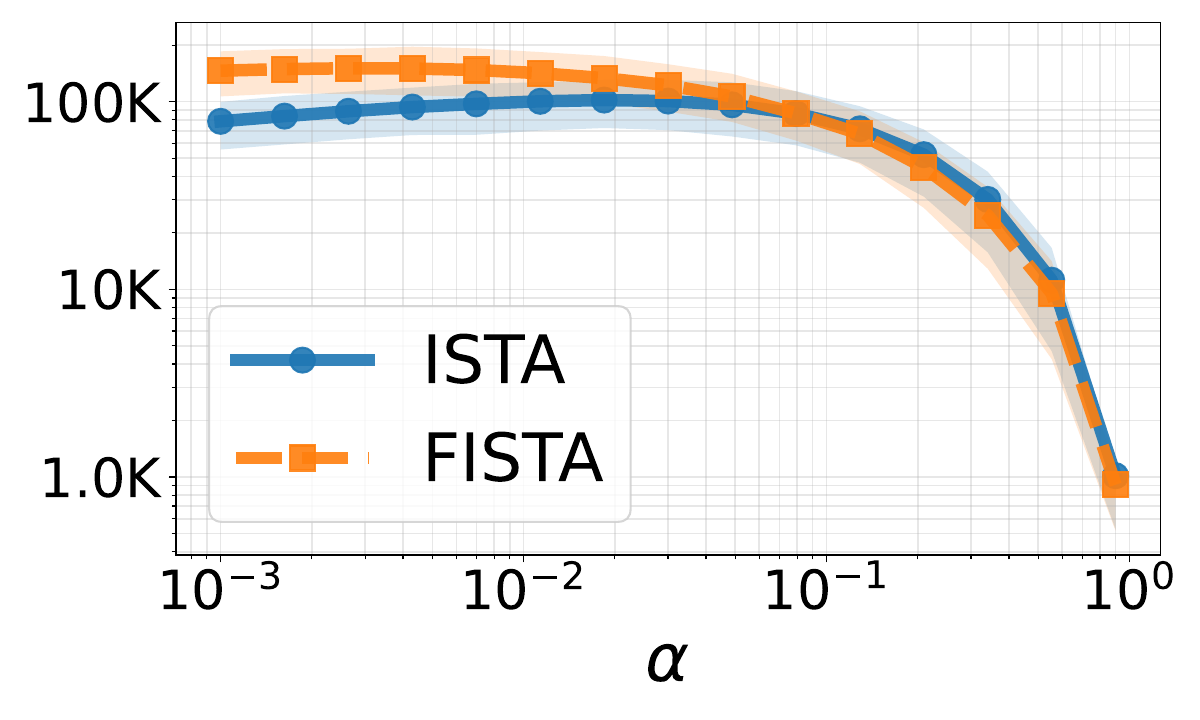}
        \caption{\texttt{com-Orkut}.}
        \label{fig:real_work_vs_alpha_orkut}
    \end{subfigure}
    \caption{\textit{Real graphs: work vs.\ $\alpha$.}
    Work to reach tolerance $10^{-8}$ as a function of $\alpha$,
    with $\rho=10^{-4}$ fixed. Curves show mean over $300$ random seeds; shaded bands are interquartile ranges.}
    \label{fig:real_work_vs_alpha}
\end{figure}

\textbf{Work vs.\ KKT tolerance $\varepsilon$.}
We next fix $\alpha=0.20$ and $\rho=10^{-4}$ and sweep the tolerance
$\varepsilon$ over a log-spaced grid in $[10^{-8},10^{-2}]$; results are in \Cref{fig:real_work_vs_epsilon}.
Tightening the tolerance (smaller $\varepsilon$) increases work for both methods, and FISTA typically
achieves the same tolerance with less total work on these datasets, though the gap can be small (notably on \texttt{com-Orkut})
for intermediate tolerances.
\begin{figure}[t!]
    \centering
    \begin{subfigure}[t]{0.25\linewidth}
        \centering
        \includegraphics[width=\linewidth]{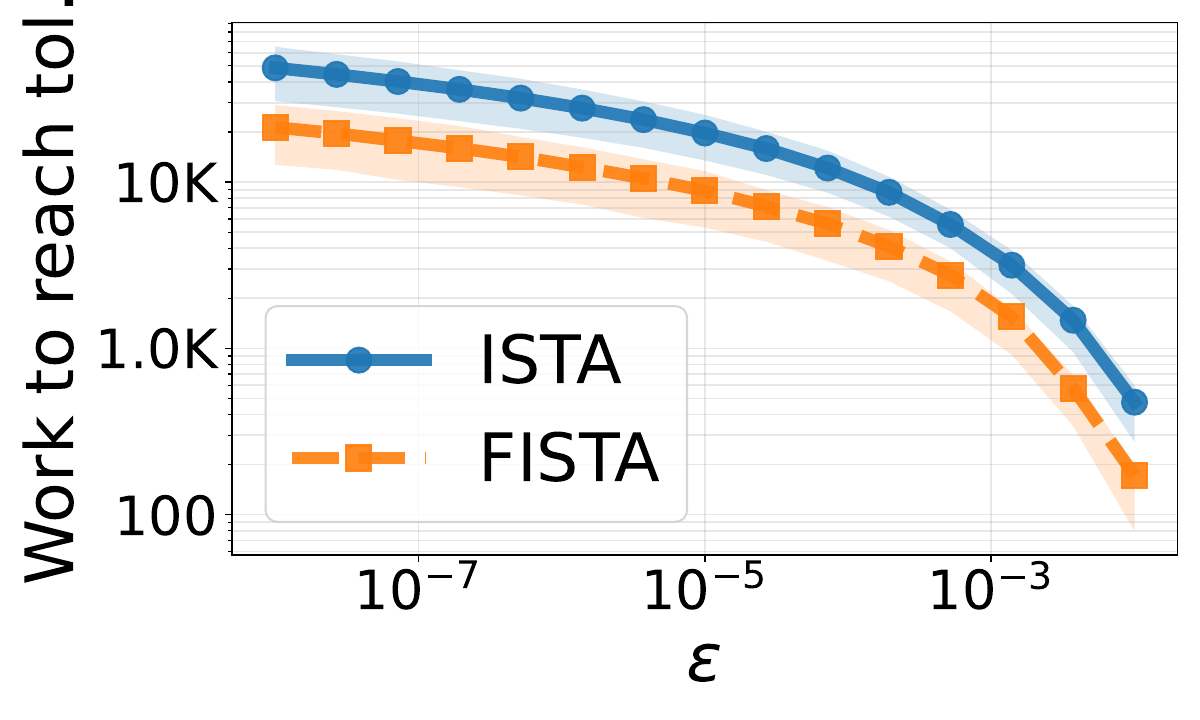}
        \caption{\texttt{com-Amazon}.}
        \label{fig:real_work_vs_epsilon_amazon}
    \end{subfigure}\hfill
    \begin{subfigure}[t]{0.25\linewidth}
        \centering
        \includegraphics[width=\linewidth]{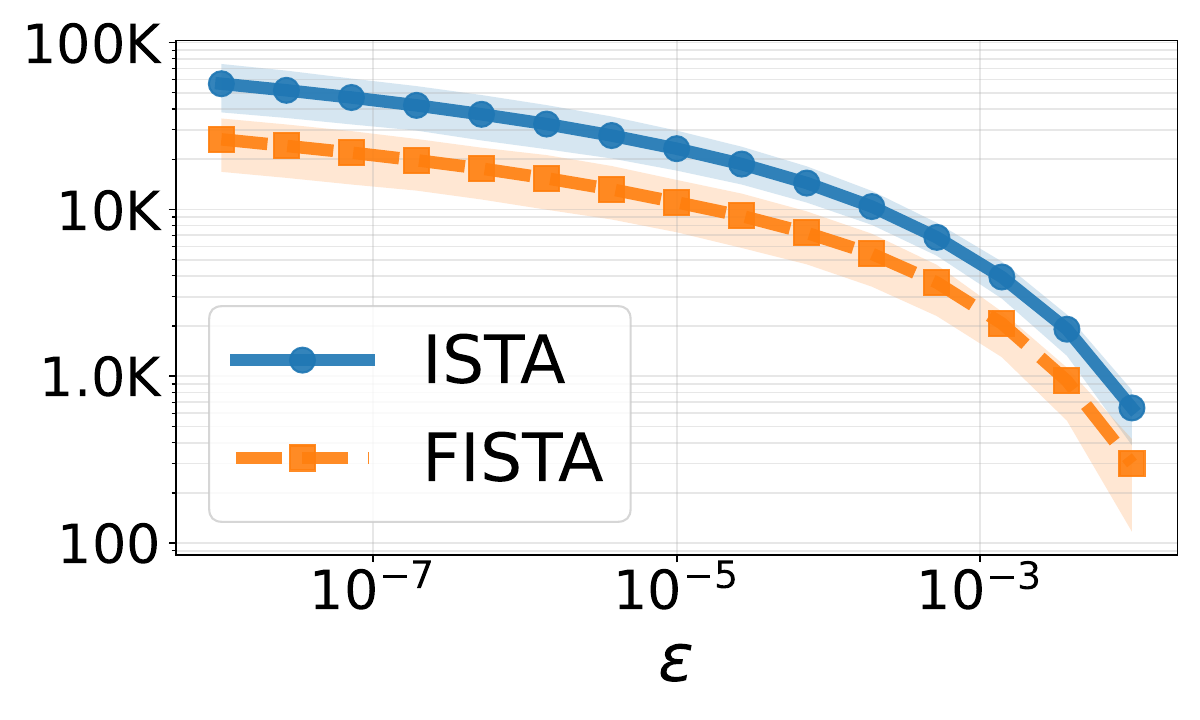}
        \caption{\texttt{com-DBLP}.}
        \label{fig:real_work_vs_epsilon_dblp}
    \end{subfigure}\hfill
    \begin{subfigure}[t]{0.25\linewidth}
        \centering
        \includegraphics[width=\linewidth]{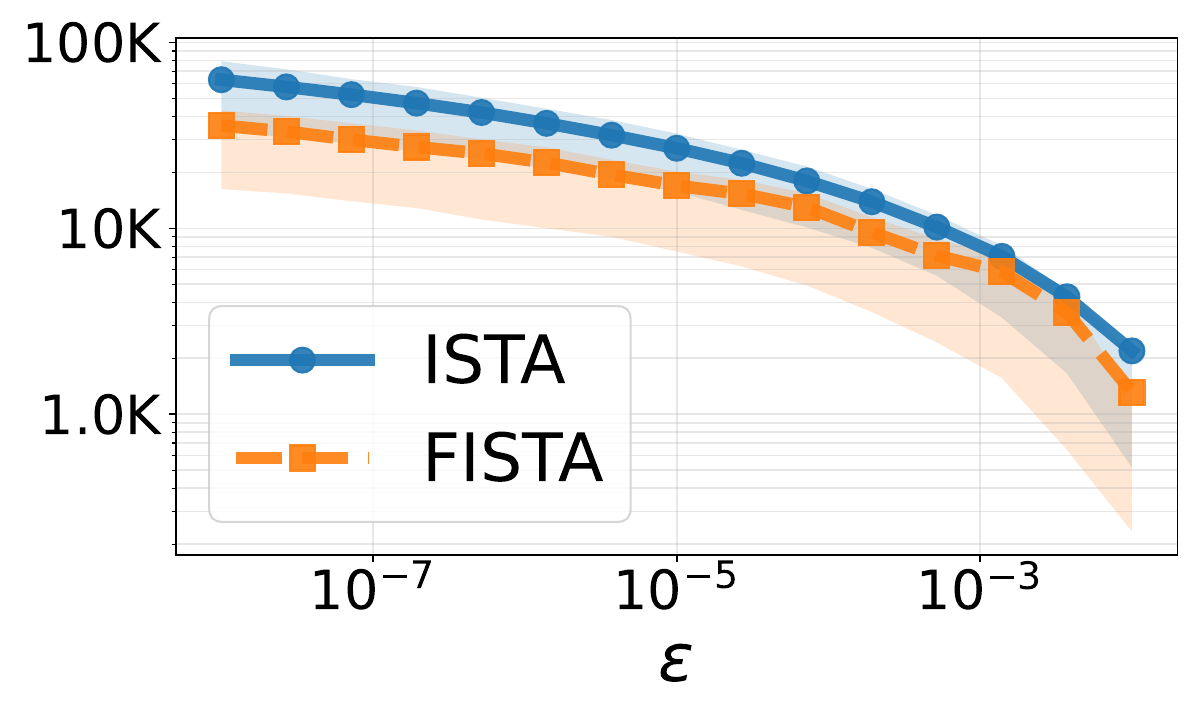}
        \caption{\texttt{com-Youtube}.}
        \label{fig:real_work_vs_epsilon_youtube}
    \end{subfigure}\hfill
    \begin{subfigure}[t]{0.25\linewidth}
        \centering
        \includegraphics[width=\linewidth]{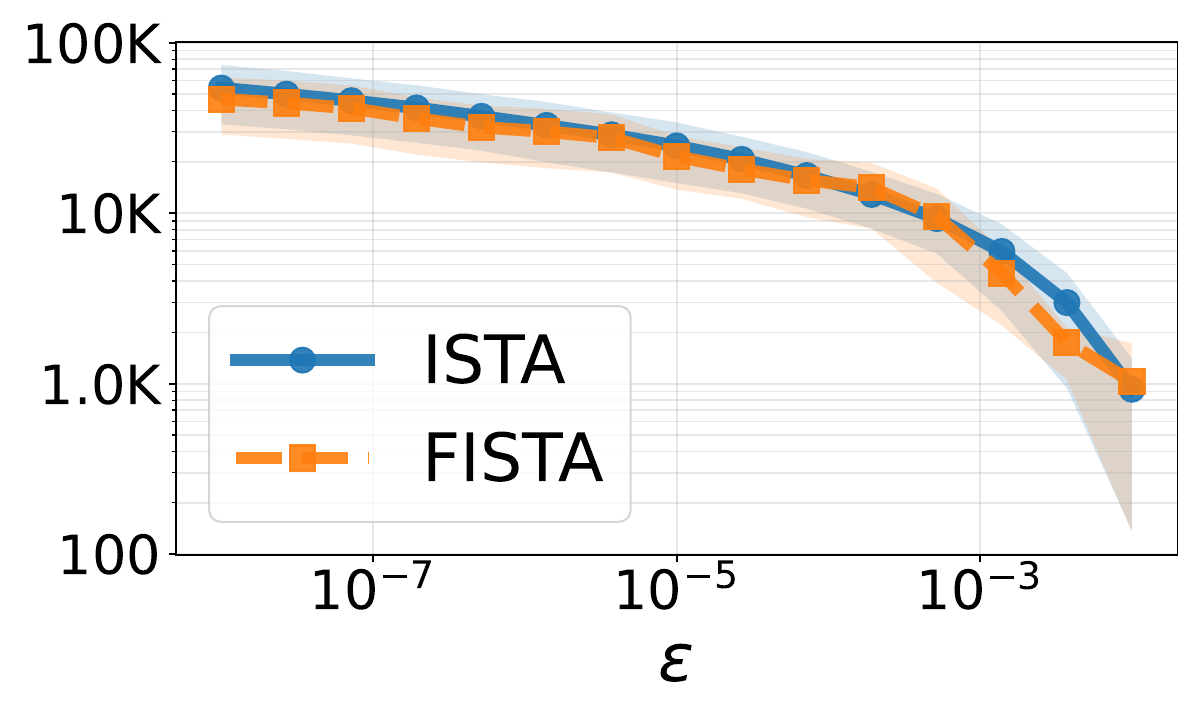}
        \caption{\texttt{com-Orkut}.}
        \label{fig:real_work_vs_epsilon_orkut}
    \end{subfigure}
    \caption{\textit{Real graphs: work vs.\ KKT tolerance.}
    Work to reach $\varepsilon$,
    with $\alpha=0.20$ and $\rho=10^{-4}$ fixed. Curves show mean over $300$ random seeds; shaded bands are interquartile ranges.}
    \label{fig:real_work_vs_epsilon}
\end{figure}

\begin{figure}[t!]
    \centering
    \begin{subfigure}[t]{0.25\linewidth}
        \centering
        \includegraphics[width=\linewidth]{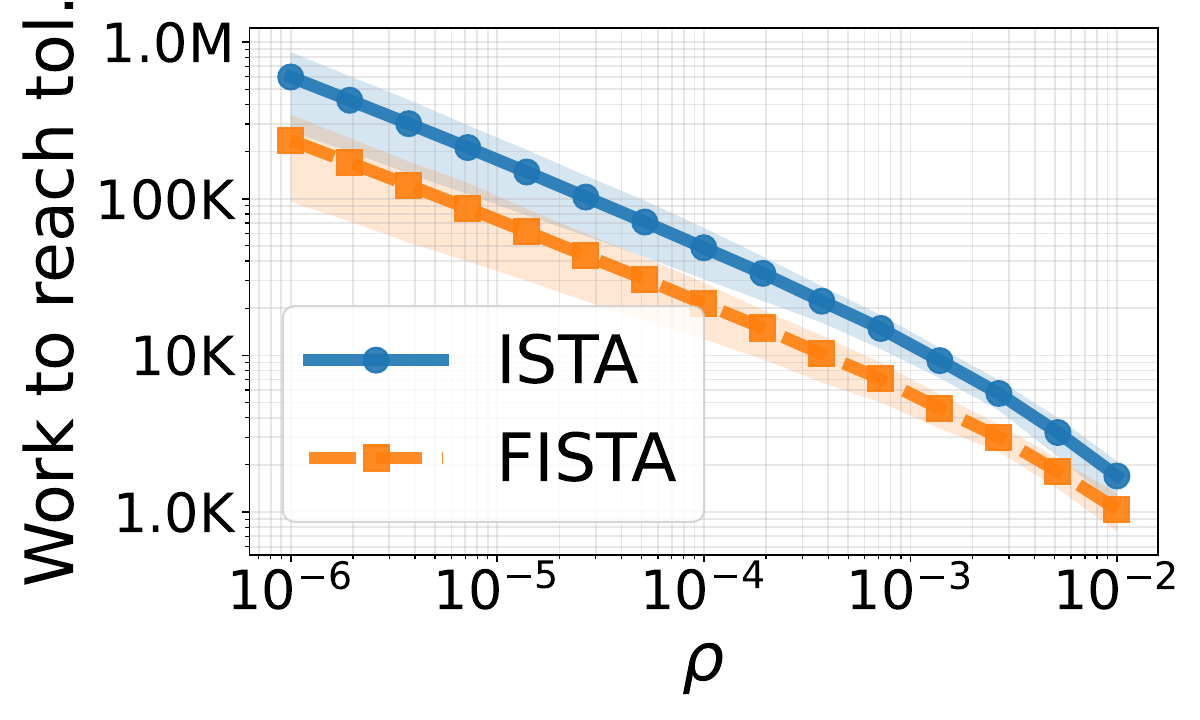}
        \caption{\texttt{com-Amazon}.}
        \label{fig:real_work_vs_rho_amazon}
    \end{subfigure}\hfill
    \begin{subfigure}[t]{0.25\linewidth}
        \centering
        \includegraphics[width=\linewidth]{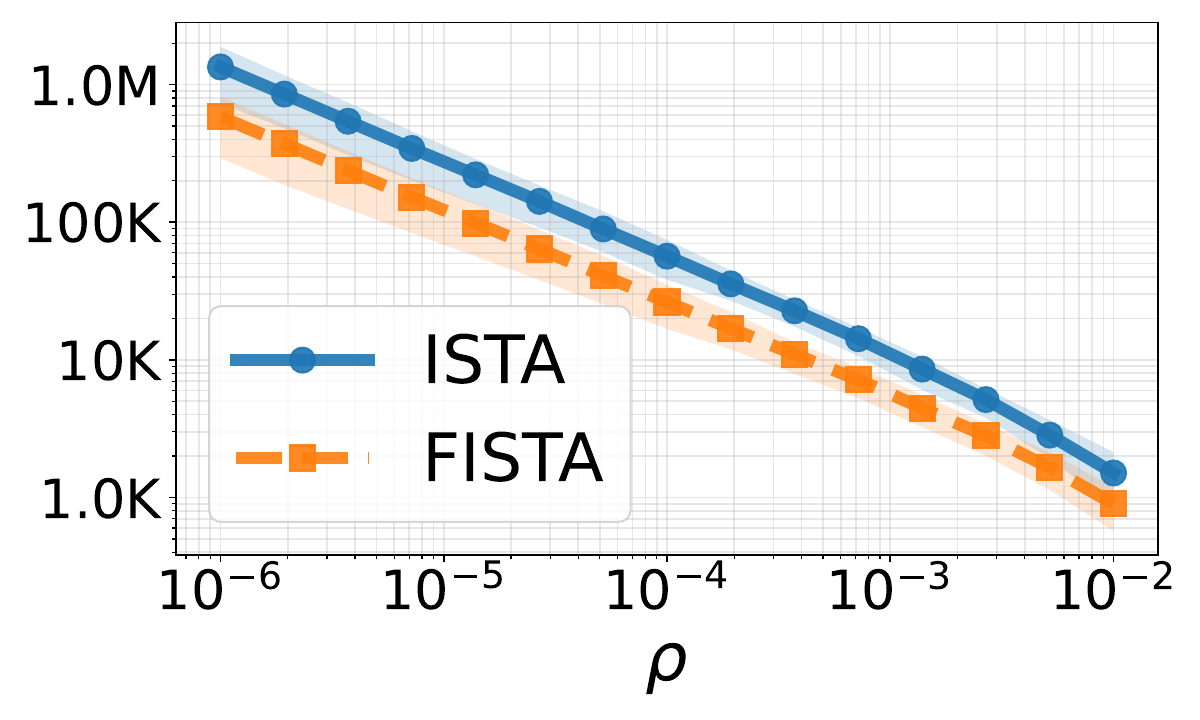}
        \caption{\texttt{com-DBLP}.}
        \label{fig:real_work_vs_rho_dblp}
    \end{subfigure}\hfill
    \begin{subfigure}[t]{0.25\linewidth}
        \centering
        \includegraphics[width=\linewidth]{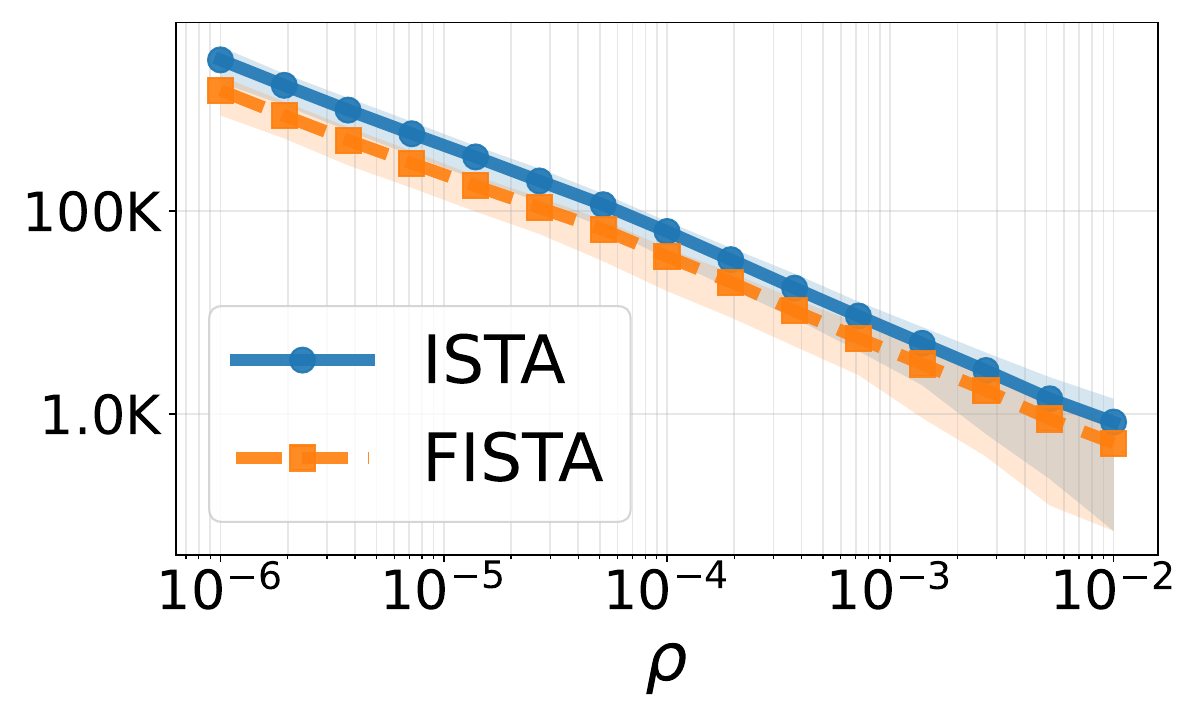}
        \caption{\texttt{com-Youtube}.}
        \label{fig:real_work_vs_rho_youtube}
    \end{subfigure}\hfill
    \begin{subfigure}[t]{0.25\linewidth}
        \centering
        \includegraphics[width=\linewidth]{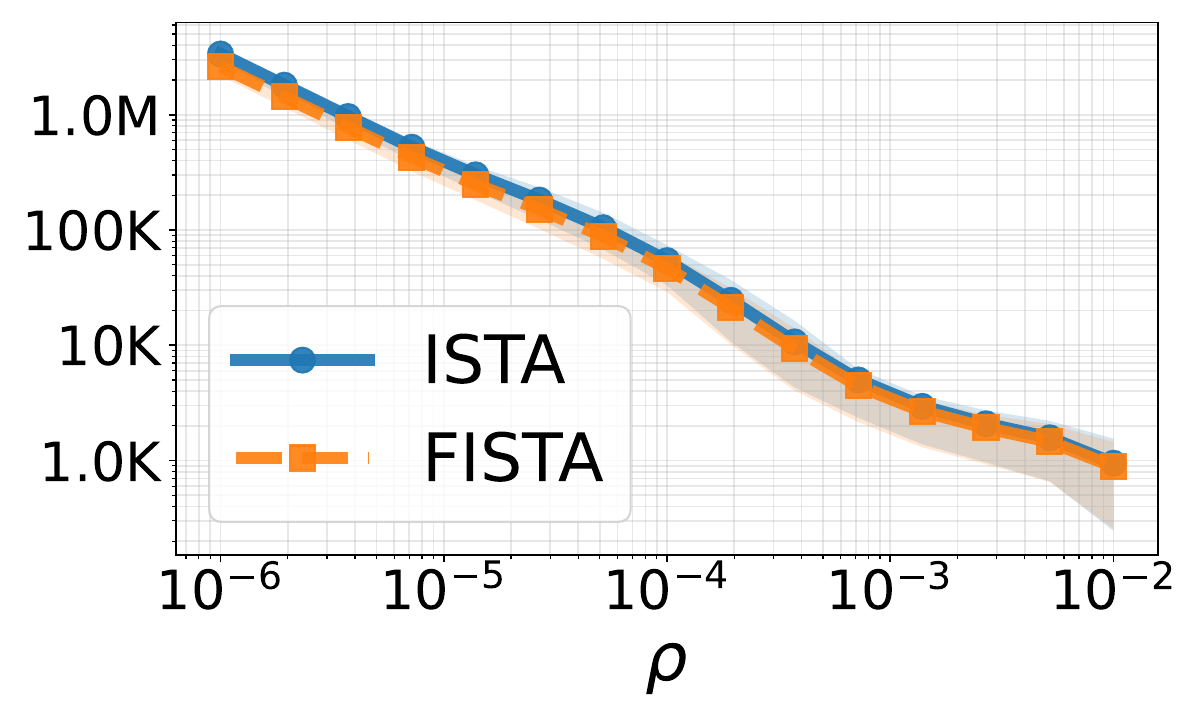}
        \caption{\texttt{com-Orkut}.}
        \label{fig:real_work_vs_rho_orkut}
    \end{subfigure}
    \caption{\textit{Real graphs: work vs.\ $\rho$.}
    Work to reach $10^{-8}$ as a function of $\rho$, with $\alpha=0.20$ fixed.
    Curves show mean over $300$ random seeds; shaded bands are interquartile ranges.}
    \label{fig:real_work_vs_rho}
\end{figure}

\textbf{Work vs.\ sparsity parameter $\rho$.}
Finally, we fix $\alpha=0.20$ and $\varepsilon=10^{-8}$ and sweep $\rho$ over a log-spaced grid in
$[10^{-6},10^{-2}]$; results are shown in \Cref{fig:real_work_vs_rho}.
As $\rho$ increases (stronger regularization), the solutions become more localized and the work decreases sharply.
Across all four graphs, FISTA generally improves upon ISTA by a modest constant factor for these settings.

\textbf{Additional diagnostics.}
Because total work conflates iteration count and per-iteration cost, aggregate curves alone do not explain why the ISTA--FISTA ranking differs across datasets. In \cref{app:real_diagnostics} we separate these two factors and we report degree-tail summaries. In particular, the diagnostics show that \texttt{com-Orkut}'s slowdowns are driven by costly transient exploration, i.e., small sets of high-degree activations inflate the per-iteration work and can offset (or even reverse) the iteration savings of acceleration
(\Cref{fig:real_tradeoff_iters_cost,fig:real_degree_ccdf}).

\section{Conclusion, limitations and future work}
We analyzed classical FISTA for
$\ell_1$-regularized PageRank under a degree-weighted locality work model. For the slightly over-regularized objective, under an explicit confinement
assumption, the resulting complexity decomposes into acceleration together with an explicit
overhead that quantifies momentum-induced transient exploration. We also provide a lower bound for the total work of standard FISTA on the original objective, based on a family of bad instances. Overall, we provide a comprehensive understanding of the behavior of FISTA on PageRank and regimes where it yields advantages. Additional work could extend our framework to other algorithms.

\ifarxiv

\section*{Acknowledgements}
K.~Fountoulakis would like to acknowledge the support of the Natural Sciences and Engineering Research Council of Canada (NSERC). Cette recherche a \'et\'e financ\'ee par le Conseil de recherches en sciences naturelles et en g\'enie du Canada (CRSNG),
[RGPIN-2019-04067, DGECR-2019-00147].

D.~Martínez-Rubio was partially funded by the Spanish Ministry of Science, Innovation, and Universities and by the State Research Agency (MICIU/AEI/10.13039/501100011033/) under grant PID2024-160448NA-I00. He was also funded by La Caixa Junior Leader Fellowship 2025.

\else

\fi

\printbibliography[heading=bibintoc] %

\clearpage

\appendix

\section{Proofs}\label{sec:proofs}

\begin{lemma}[Initial gap]\label{lemma:initial_gap}
Assume the seed is a single node $v$ so $s=e_v$ and initialize $x_0=0$.
Then $F_\rho(0)=0$ and
\[
\Delta_0 = F_\rho(0)-F_\rho(x^\star) = -F_\rho(x^\star) \le \frac{\alpha}{2d_v}\le \frac{\alpha}{2}.
\]
\end{lemma}
\begin{proof}
We have $F_\rho(x)\ge f(x)$ and thus
$F_\rho(x^\star)\ge \min_x f(x)$.
The unconstrained minimizer of the quadratic $f$ is $x_f^\star=Q^{-1}b$ and $\min f = -\frac12 b^\top Q^{-1}b$.
Because $Q\succeq \alpha I$, we have $Q^{-1}\preceq \frac{1}{\alpha} I$, and therefore
\[
b^\top Q^{-1}b \;\le\; \frac{1}{\alpha}  b^\top b \;=\; \frac{1}{\alpha}\alpha^2 \|D^{-1/2}s\|_2^2 \;=\; \alpha s^\top D^{-1}s.
\]
For $s=e_v$, we have $s^\top D^{-1}s = 1/d_v$, which yields
\[
\min f \ge -\frac{\alpha}{2d_v},
\qquad
\text{and hence}
\qquad
-F_\rho(x^\star)\le -\min f \le \frac{\alpha}{2d_v}.
\]
\end{proof}

We now state the classical guarantee on the function-value convergence on FISTA.

\begin{fact}[FISTA convergence rate]\label{fact:fista_rate}
Assume $f$ is $L$-smooth and $F$ is $\mu$-strongly convex with respect to $\norm{\cdot}_2$.
Run \cref{eq:apg_fista} with $\eta = 1 / L$ and
$\beta \defi \frac{\sqrt{L/\mu} - 1}{\sqrt{L / \mu}+1} \in [0, 1)$.
Then
\[
F(x_k) - F(x^\star)
\leq
2\left(F(x_0)-F(x^\star) \right) \left(1 - \sqrt{\frac{\mu}{L}}\right)^k
\qquad \text{for all } k \geq 1.
\]
\end{fact}
See \citep[Section 10.7.7]{beck2017first} and take into account that
$\frac{\mu}{2}\norm{x_0-x^\star}_2^2 \leq F(x_0)-F(x^\star)$ by strong convexity.
We note that the convergence guarantee above, along with strong convexity, yields bounds on the distance to the minimizer $x^\star$.

As a corollary, we can bound the distance to optimizer of the iterates along the whole computation path.
\begin{corollary}[FISTA iterates]\label{lem:yk_distance}%
In the setting of \cref{fact:fista_rate}, we have:
\begin{equation*}
\|y_0-x^\star\|_2^2 \le \frac{4\Delta_0}{\mu}
\quad \text{ and }\quad
\|y_k-x^\star\|_2^2 \le M\left(1-\sqrt{\frac{\mu}{L}}\right)^{k-1}
\quad \text{ for all } k \geq 1,
\end{equation*}
for $M \defi \frac{8\Delta_0}{\mu}\Bigl((1+\beta)^2 (1-\sqrt{\mu / L}) + \beta^2\Bigr)$.
\end{corollary}

Using the bounds on $\Delta_0$ and $\mu$ in \cref{sec:preliminaries}, one obtains that for \cref{eq:reg_personalized_pagerank}
it is $M \leq 20$ and $\Delta_0 / \mu \leq 1/2$.
Thus $\norm{y_k - x^\star}_2 \leq \sqrt{20}$ for all $k\geq 0$.

\begin{proof}%
Let $q \defi 1 - \sqrt{\mu / L}$. Strong convexity gives $F(x)-F(x^\star)\ge \frac{\mu}{2}\|x-x^\star\|_2^2$, hence, by \cref{fact:fista_rate}:
\[
\|x_k-x^\star\|_2^2 \le \frac{2}{\mu}(F(x_k)-F(x^\star)) \le \frac{4\Delta_0}{\mu}q^k,
\]
which already yields the result for $k=0$, using $y_0=x_0$. For $k\ge 1$, write $y_k-x^\star=(1+\beta)(x_k-x^\star)-\beta(x_{k-1}-x^\star)$ and use $ \|a-b\|_2^2 \le 2\|a\|_2^2 + 2\|b\|_2^2 $ to obtain
\[
\|y_k-x^\star\|_2^2 \le 2(1+\beta)^2\|x_k-x^\star\|_2^2 + 2\beta^2\|x_{k-1}-x^\star\|_2^2.
\]
Substitute the bounds on $\|x_k-x^\star\|_2^2$ and $\|x_{k-1}-x^\star\|_2^2$.
\end{proof}

\begin{proof}\linkofproof{lem:coord_jump}
Fix $i\in A(y)\subseteq I^\star$. Then $x_i^\star=0$, and we have
\begin{align*}
\begin{aligned}
    \abs{u(y)_i-u(x^\star)_i} \circled{1}[\ge] \absl{\abs{u(y)_i}-\abs{u(x^\star)_i}}
    \geq |u(y)_i|-|u(x^\star)_i|
    \circled{2}[>] \eta \lambda_i - \eta (\lambda_i-\gamma_i\sqrt{d_i})
    = \eta \gamma_i\sqrt{d_i}.
\end{aligned}
\end{align*}
    where we used the reverse triangle inequality in $\circled{1}$. In $\circled{2}$ we used that $i \in A(y)$ means $\prox_{\eta g}(u(y))_i \neq 0$ and so $|u(y)_i| > \eta \lambda_i$. And also that by definition, we have $\abs{u(x^\star)_i} = \eta\abs{\nabla f(x^\star)_i} = \eta(\lambda_i - \gamma_i \sqrt{d_i})$ by the definition of $\gamma_i$.
\end{proof}

The following \Cref{lem:path_monotone} is proved, for example, as Lemma~4 in \citet{ha2021statistical} (see also the discussion of regularization paths in \citet{fountoulakis2019variational}).

\begin{lemma}[Monotonicity of the $\ell_1$-regularized PageRank path \citep{ha2021statistical} ]\label{lem:path_monotone}
    For the family \cref{eq:reg_personalized_pagerank}, let $x^\star(\rho) \defi \argmin_{x} F_\rho(x)$, for any $\rho > 0$.  The solution path is monotone: if $\rho'>\rho \geq 0$, then
\begin{equation*}
x^\star(\rho')\le x^\star(\rho)\quad\text{coordinatewise.}
\end{equation*}
\end{lemma}

\begin{lemma}[Monotonicity of proximal gradient steps for PageRank]\label{lem:prox_grad_monotonicity}
If $z \geq z' \geq 0$, then
\[
\prox_{g_c}\!\left(z - \frac{1}{L} \nabla f(z)\right)
\;\geq\;
\prox_{g_c}\!\left(z' - \frac{1}{L}\nabla f(z')\right)
\qquad \text{(componentwise).}
\]
\end{lemma}

\begin{proof}
Using the definition of the forward-gradient map $u(z)\defi z-\frac{1}{L}\nabla f(z)$, and that, for PageRank, $\nabla f(z)=Qz-b$ with $b=\alpha D^{-1/2}s$, we have
\[
u(z) = z-\frac{1}{L}(Qz-b) = \left(I-\frac{1}{L}Q\right)z + \frac{1}{L}b.
\]
Since $Q$ is an $M$-matrix (positive semidefinite and off-diagonal entries are nonpositive) and $Q\preccurlyeq LI$ (by $L$-smoothness),
the matrix $I-\frac{1}{L}Q$ is entrywise nonnegative. Therefore $u(\cdot)$ is monotone componentwise:
$z\ge z'\Rightarrow u(z)\ge u(z')$.

Next, for $g_c(x)=c\alpha\|D^{1/2}x\|_1$, the proximal map is separable and monotone componentwise, since
\[
\bigl(\prox_{g_c}(w)\bigr)_i=\operatorname{sign}(w_i)\max\{|w_i|-c\alpha\sqrt{d_i},0\}.
\]
Composing these two monotone maps yields the claim.
\end{proof}

\begin{proof}\linkofproof{lem:two_tier_split}
Fix any $i\in I_B^{\rm small}\subseteq I_B$. Then $x_{B,i}^\star=0$.
Recall that $u(z)\defi z-\nabla f(z)$ (here $\eta=1$).

By the definition \cref{def:margin} of the (B)-margin at coordinate $i$ and the fact that $x_{B,i}^\star=0$,
we have
\begin{align*}
\gamma^{(B)}_i < \rho\alpha
&\Longleftrightarrow
2\rho\alpha + \frac{\nabla_i f(x_B^\star)}{\sqrt{d_i}} < \rho\alpha \\
&\Longleftrightarrow
\frac{\nabla_i f(x_B^\star)}{\sqrt{d_i}} < -\,\rho\alpha \\
&\Longleftrightarrow
\nabla_i f(x_B^\star) < -\,\rho\alpha\sqrt{d_i} \\
&\Longleftrightarrow
-\nabla_i f(x_B^\star) > \rho\alpha\sqrt{d_i} \\
&\Longleftrightarrow
u(x_B^\star)_i > \rho\alpha\sqrt{d_i}.
\end{align*}
Therefore, applying the coordinate formula for the prox of $g_A$ (cf.\ \cref{eq:prox_in_fista}) gives
\[
\bigl(\prox_{g_A}(u(x_B^\star))\bigr)_i
=
\operatorname{sign}(u(x_B^\star)_i)\max\{|u(x_B^\star)_i|-\rho\alpha\sqrt{d_i},\,0\}
>0.
\]
Next, since $\rho_B=2\rho>\rho_A=\rho$, path monotonicity \cref{lem:path_monotone} yields
$x_A^\star \ge x_B^\star$ componentwise.
Applying \cref{lem:prox_grad_monotonicity} with $c=1$ (i.e., $g_A$), $z=x_A^\star$, and $z'=x_B^\star$, we obtain
\[
\prox_{g_A}(u(x_A^\star)) \;\ge\; \prox_{g_A}(u(x_B^\star))
\qquad\text{(componentwise).}
\]
Finally, $x_A^\star$ is a fixed point of the proximal-gradient map for the (A) problem, so
$x_A^\star=\prox_{g_A}(u(x_A^\star))$. Hence
\[
x_{A,i}^\star \;=\; \bigl(\prox_{g_A}(u(x_A^\star))\bigr)_i
\;\ge\; \bigl(\prox_{g_A}(u(x_B^\star))\bigr)_i
\;>\;0,
\]
which implies $i\in \supp(x_A^\star)=S_A$. Therefore $I_B^{\rm small}\subseteq S_A$.
\end{proof}

\begin{proof}\linkofproof{thm:double_reg_work}
Recall that $\widetilde{A}_k=\supp(x_{k+1})\cap S_A^c$ and that we assume $\widetilde{A}_k\subseteq\mathcal{B}$ for all $k\ge0$.
By \cref{lem:two_tier_split}, any index in $\widetilde{A}_k$ is (B)-inactive with margin at least $\rho\alpha$; concretely,
    $\gamma_i^{(B)}\ge \rho\alpha$ for all $i\in\widetilde{A}_k$. Let $\mathcal{B}'$ be the subset of $\mathcal{B}$ such that $\gamma_i^{(B)} \geq \rho \alpha$ for all $i \in \mathcal{B}'$. Thus, $\widetilde{A}_k \subseteq \mathcal{B}'$.

We first bound the total spurious work:
\begin{align}\label{eq:spurious_work}
\begin{aligned}
\sum_{k=0}^{\infty}\sum_{i\in\widetilde{A}_k} d_i
&=
\sum_{k=0}^{\infty}\sum_{i\in\widetilde{A}_k}
\left(\frac{\sqrt{d_i}}{\gamma_i^{(B)}}\right)\left(\gamma_i^{(B)}\sqrt{d_i}\right) \\
&\le
\sum_{k=0}^{\infty}
\sqrt{\sum_{i\in\mathcal{B}'}\frac{d_i}{(\gamma_i^{(B)})^2}}
\;
\sqrt{\sum_{i\in\widetilde{A}_k}(\gamma_i^{(B)})^2 d_i}
\qquad \text{(Cauchy-Schwarz, and } \widetilde{A}_k\subseteq\mathcal{B}')\\
&\le
\sum_{k=0}^{\infty}
\sqrt{\sum_{i\in\mathcal{B}' }\frac{d_i}{\rho^2\alpha^2}}
\;
\sqrt{\sum_{i\in\widetilde{A}_k} |u(y_k)_i-u(x_B^\star)_i|^2}
\qquad \text{(since } \gamma_i^{(B)}\ge\rho\alpha,\ \cref{lem:coord_jump})\\
&\le
\sum_{k=0}^{\infty}
\frac{\sqrt{\vol(\mathcal{B})}}{\rho\alpha}\;
    \|u(y_k)-u(x_B^\star)\|_2 \qquad (\text{because } \mathcal{B}' \subseteq \mathcal{B})\\
&\le
\sum_{k=0}^{\infty}
    \frac{\sqrt{\vol(\mathcal{B})}}{\rho\alpha}\; 
\left(\|y_k-x_B^\star\|_2+\eta\|\nabla f(y_k)-\nabla f(x_B^\star)\|_2\right) \\
&\le
\sum_{k=0}^{\infty}
\frac{\sqrt{\vol(\mathcal{B})}}{\rho\alpha}\;
    (1+\eta L)\|y_k-x_B^\star\|_2. \qquad \text{(by smoothness)}
\end{aligned}
\end{align}
For RPPR we use $\eta=1$ and $L=1$, hence $(1+\eta L)=2$.
Using \cref{lem:yk_distance} and writing $q\defi 1-\sqrt{\mu/L}=1-\sqrt{\alpha}$, we have
$\|y_k-x_B^\star\|_2\le O(1)\,q^{k/2}$, so the series in \eqref{eq:spurious_work} sums to $O((1-q)^{-1})=O(\alpha^{-1/2})$.
Therefore,
\[
\sum_{k=0}^{\infty}\vol(\widetilde{A}_k)
\;=\;
O\!\left(\frac{\sqrt{\vol(\mathcal{B})}}{\rho\alpha^{3/2}}\right).
\]

Next, we bound the core work. By the sparsity guarantee for RPPR \citep[Theorem~2]{fountoulakis2019variational},
$\vol(S_A)\le 1/\rho$. Thus, after $N$ iterations, the total work is bounded by
\[
\mathrm{Work}(N)
=
O\!\left(N\vol(S_A) + \sum_{k=0}^{N-1}\vol(\widetilde{A}_k)\right)
=
O\!\left(\frac{N}{\rho} + \frac{\sqrt{\vol(\mathcal{B})}}{\rho\alpha^{3/2}}\right).
\]
Finally, by \cref{fact:fista_rate}, to ensure $F_B(x_N)-F_B(x_B^\star)\le \varepsilon$ it suffices to take
\[
N \ge N_\varepsilon \defi
\left\lceil \frac{\log(\Delta_0/\varepsilon)}{\log\!\left(1/(1-\sqrt{\mu/L})\right)}\right\rceil
=
O\!\left(\frac{1}{\sqrt{\alpha}}\log\!\left(\frac{\alpha}{\varepsilon}\right)\right),
\]
using $L=1$, $\mu=\alpha$, and $\Delta_0\le \alpha/2$ from \cref{sec:preliminaries}.
Substituting $N=N_\varepsilon$ yields the stated bound.
\end{proof}

\begin{proof}\linkofproof{thm:boundary_confinement_exposure}
We prove by induction that $\supp(x_k)\subseteq S\cup \bdry S$ for all $k\ge 0$. The base case is trivial, since $x_0=0$, so $\supp(x_0)=\emptyset\subseteq S\cup \bdry S$.

Now assume $\supp(x_{k-1})\subseteq S\cup \bdry S$ and $\supp(x_k)\subseteq S\cup \bdry S$ for some $k\ge 0$
(with $x_{-1}=x_0=0$ covering $k=0$).
Then $\supp(y_k)\subseteq \supp(x_k)\cup\supp(x_{k-1})\subseteq S\cup \bdry S$.

Fix any $i\in \mathrm{Ext}(S)$.
Since $i\notin S\cup\bdry S$ and $\supp(y_k)\subseteq S\cup\bdry S$, we have $y_{k,i}=0$.
Also $i\neq v$ (the seed lies in $S$), hence $(D^{-1/2}s)_i=0$.

For RPPR, $\nabla f(y)=Qy-\alpha D^{-1/2}s$, so with $\eta=1$ we have
\[
u(y)=y-\nabla f(y)=(I-Q)y+\alpha D^{-1/2}s.
\]
Using $Q=\frac{1+\alpha}{2}I-\frac{1-\alpha}{2}D^{-1/2}AD^{-1/2}$, we get
$(I-Q)=\frac{1-\alpha}{2}\left(I+D^{-1/2}AD^{-1/2}\right)$.
Therefore, for our fixed $i\in\mathrm{Ext}(S)$,
\[
u(y_k)_i
=
\frac{1-\alpha}{2}\sum_{j\sim i}\frac{y_{k,j}}{\sqrt{d_i d_j}}
=
\frac{1-\alpha}{2}\sum_{j\in \N(\{i\})\cap \bdry S}\frac{y_{k,j}}{\sqrt{d_i d_j}},
\]
because $i$ has no neighbors in $S$ (by definition of $\mathrm{Ext}(S)$) and $y_k$ has no support outside $S\cup\bdry S$. Taking absolute values and using $d_j\ge d_{\min\bdry S}$ for $j\in\bdry S$ gives
\begin{align*}\label{eq:u_bound}
\begin{aligned}
|u(y_k)_i|
&\le
\frac{1-\alpha}{2\sqrt{d_i}}
\sum_{j\in \N(\{i\})\cap \bdry S}\frac{|y_{k,j}|}{\sqrt{d_j}}
    \circled{1}[\le]
\frac{1-\alpha}{2\sqrt{d_i}}
\cdot
\frac{\sqrt{|\N(\{i\})\cap \bdry S|}}{\sqrt{d_{\min\bdry S}}}
\cdot
\|y_k\|_2 \\
    & \circled{2}[\leq] \frac{\alpha}{4}\rho\sqrt{d_i} (\norm{y_k-x^\star}_2 + \norm{x^\star}_2) \circled{3}[\leq] 2\alpha\rho\sqrt{d_i}.
\end{aligned}
\end{align*}

where $\circled{1}$ uses Cauchy-Schwarz and $\circled{2}$ uses the triangular inequality and our no-percolation assumption \cref{eq:exposure_assumption}. Finally $\circled{3}$ uses the bound $\norm{y_k-x^\star}_2 \leq \sqrt{20}$ from \cref{lem:yk_distance} and $\norm{x^\star}_2 \leq 1$ from the preliminaries \cref{sec:preliminaries} and bound the resulting constants to an integer number.

For the (B) problem, the shrinkage threshold is $\lambda_i=2\alpha\rho\sqrt{d_i}$, so we showed 
$|u(y_k)_i|\le \lambda_i$ and thus the proximal update keeps $x_{k+1,i}=0$ (cf.\ the coordinate formula \cref{eq:prox_in_fista}).
Since this holds for every $i\in\mathrm{Ext}(S)$, we conclude $\supp(x_{k+1})\subseteq S\cup\bdry S$. This completes the induction.
\end{proof}

\section{High-degree nodes do not activate}
\label{sec:fista_degree_nonactivation}
We provide an extra property of \cref{eq:apg_fista}. Nodes of high-enough degree can never be activated by the accelerated iterates when we over-regularize. The proof argues that \cref{eq:apg_fista} on problem (B) with a large margin prevents high-degree nodes, and large margins are guaranteed for coordinates outside $S_A$.

\begin{proposition}[Large-degree nodes do not activate]\label{lem:fista_degree_cutoff}%
    Run \Cref{eq:apg_fista} on problem (B) $F_{2\rho}$, and let $R$ be any uniform bound such that $\|y_k-x_B^\star\|_2\le R$ for all $k\ge 0$.
    Let $i$ be such that the minimizer of problem (A) $F_{\rho}$ satisfies $x_{A,i}^\star=0$.
If
\begin{equation*}\label{eq:rho_gap_condition_degree}
    d_i \ge  \left(\frac{LR}{\alpha\rho} \right)^2,
\end{equation*}
then \Cref{eq:apg_fista} does not activate node $i$, that is $x_{k,i}=y_{k,i}=0$ for all $k\ge 0$.
\end{proposition}

For our PageRank problem \cref{eq:reg_personalized_pagerank}, we have $L \leq 1$, and \Cref{lem:yk_distance} allows us to take $R\le \sqrt{20}$.
Therefore nodes satisfying $d_i \geq 20\alpha^{-2} \rho^{-2}$ will not get activated.

\begin{proof}
Fix $i$ such that $x_{A,i}^\star=0$.
By path monotonicity \Cref{lem:path_monotone} and nonnegativity of the minimizers,
\[
\Delta x \defi x_A^\star-x_B^\star \ge 0
\qquad\text{and}\qquad
\Delta x_i = 0.
\]

Recall $\nabla f(x)=Qx-b$, where $b\ge 0$ and $Q_{ij}\le 0$ for $i\neq j$.
Since $x_A^\star,x_B^\star\ge 0$ and $x_{A,i}^\star=x_{B,i}^\star=0$, we have
\[
\nabla_i f(x_A^\star)\le 0,
\qquad
\nabla_i f(x_B^\star)\le 0,
\]
and
\[
\nabla_i f(x_B^\star)-\nabla_i f(x_A^\star)
=
-(Q\Delta x)_i
=
-\sum_{j\neq i} Q_{ij}\Delta x_j
\ge 0.
\]
Hence
\[
|\nabla_i f(x_B^\star)|
\le
    |\nabla_i f(x_A^\star)| \leq \alpha \rho \sqrt{d_i},
\]
where the last inequality holds by the KKT conditions for problem (A) ($\rho$-regularization),

We now prove $x_{k,i}=0$ for all $k$ by induction.
The base case holds since $x_{-1}=x_0=0$.
Assume $x_{k,i}=x_{k-1,i}=0$.
Then $y_{k,i}=0$ and by $L$-smoothness,
\[
|\nabla_i f(y_k)|
\le
|\nabla_i f(x_B^\star)|
+
\|\nabla f(y_k)-\nabla f(x_B^\star)\|_2
\le
\alpha\rho\sqrt{d_i}
+
L\|y_k-x_B^\star\|_2.
\]
Using $\|y_k-x_B^\star\|_2\le R$ and
\(
d_i \ge \left(\frac{LR}{\alpha\rho}\right)^2
\),
we obtain
\begin{equation}\label{eq:compact_grad_yk}
|\nabla_i f(y_k)|
\le
2\alpha\rho\sqrt{d_i}.
\end{equation}

Let $w_k=y_k-\eta\nabla f(y_k)$.
Since $y_{k,i}=0$,
\(
|w_{k,i}|=\eta|\nabla_i f(y_k)|.
\)
The proximal map of $g(x)=\alpha\rho\|D^{1/2}x\|_1$ is weighted soft-thresholding, so by \eqref{eq:compact_grad_yk},
\[
x_{k+1,i}
=
\operatorname{sign}(w_{k,i})
\max\{ |w_{k,i}|-2\eta\alpha\rho\sqrt{d_i},0\}
=
0.
\]
This closes the induction and proves $x_{k,i}=0$ for all $k\ge 0$.
\end{proof}

\section{Bad instances where the margin \texorpdfstring{$\gamma$}{gamma} can be very small}
\label{sec:bad_instances}

We now record two explicit graph families showing that the degree-normalized
strict-complementarity margin (the one that naturally interfaces with our
degree-weighted work model in \Cref{eq:running_time}) can be made arbitrarily small
(and even $\gamma=0$) by choosing $\rho$ near a breakpoint of the regularization path where an inactive KKT inequality becomes tight. This motivates our theory for linking a problem (B) with the sparsity pattern of a slightly less regularized one (A), so that no requirement in the minimum margin is made (since we split the coordinates into low-margin ones which are included in the support of the (A) solution and then the high-margin ones). Concretely, let $x^\star$ be the minimizer and $I^\star \defi \{i:\;x_i^\star=0\}$.
Define the coordinatewise degree-normalized KKT slack
\[
\gamma_i \defi \frac{\lambda_i-|\nabla_i f(x^\star)|}{\sqrt{d_i}}
\qquad (i\in I^\star),
\qquad
\lambda_i \defi \rho\alpha \sqrt{d_i},
\]
and the global margin
\[
\gamma \defi \min_{i\in I^\star}\gamma_i.
\]

\subsection{Star graph (seed at the center)}

Let $G$ be a star on $m+1$ nodes with center node $c$ of degree $d_c=m$ and leaves $\ell$ of degree $1$.
Let the seed be $s=e_c$.

\begin{lemma}[Star graph breakpoint: $\gamma$ can be $0$]\label{lem:star_gamma_small}
Fix $\alpha\in(0,1)$ and $m\ge 1$ and define
\[
\rho_0 \defi \frac{1-\alpha}{2m}.
\]
For any $\rho\in[\rho_0,\,1/m)$, let
\[
x^\star \defi x_c^\star e_c,
\qquad
x_c^\star = \frac{2\alpha(1-\rho m)}{(1+\alpha)\sqrt{m}}.
\]
Then $x^\star$ is a minimizer of $F_\rho$, hence the unique minimizer by $\alpha$-strong convexity. In particular, $S^\star=\{c\}$.
Moreover, for any leaf $\ell$ (recall $d_\ell=1$), with $\lambda_\ell=\alpha\rho\sqrt{d_\ell}=\alpha\rho$,
the degree-normalized slack equals
\[
\gamma_\ell
\;\defi\;
\frac{\lambda_\ell-|\nabla_\ell f(x^\star)|}{\sqrt{d_\ell}}
\;=\;
\lambda_\ell-|\nabla_\ell f(x^\star)|
\;=\;
\frac{2\alpha}{1+\alpha}\,(\rho-\rho_0),
\]
and thus at the breakpoint $\rho=\rho_0$ one has $\gamma=0$.
\end{lemma}

\begin{proof}
Recall that $f(x)=\frac12 x^\top Qx-\alpha\langle D^{-1/2}s,x\rangle$, hence
\[
\nabla f(x)=Qx-\alpha D^{-1/2}s,
\]
and
\[
Q=\alpha I+\frac{1-\alpha}{2}\bigl(I-D^{-1/2}AD^{-1/2}\bigr)
=\frac{1+\alpha}{2}I-\frac{1-\alpha}{2}D^{-1/2}AD^{-1/2}.
\]
On the star with seed $s=e_c$, we have $D^{-1/2}s=e_c/\sqrt{m}$, and for each leaf $\ell$,
\(
(D^{-1/2}AD^{-1/2})_{\ell c}=1/\sqrt{m}.
\)
Therefore
\[
Q_{cc}=\frac{1+\alpha}{2},
\qquad
Q_{\ell c}=Q_{c\ell}=-\frac{1-\alpha}{2\sqrt{m}}.
\]
Assume $x^\star=x_c^\star e_c$ with $x_c^\star>0$.
For the $\ell_1$-regularized PageRank objective
$F_\rho(x)=f(x)+\alpha\rho\|D^{1/2}x\|_1$,
the coordinatewise KKT conditions (cf.\ \Cref{eq:KKT_coord_rho}) give, at an active coordinate,
\(
\nabla_c f(x^\star)=-\alpha\rho\sqrt{d_c}=-\alpha\rho\sqrt{m}.
\)
Using $\nabla_c f(x^\star)=(Qx^\star)_c-\alpha/\sqrt{m}=Q_{cc}x_c^\star-\alpha/\sqrt{m}$, we obtain
\[
0=\nabla_c f(x^\star)+\alpha\rho\sqrt{m}
=
Q_{cc}x_c^\star-\frac{\alpha}{\sqrt{m}}+\alpha\rho\sqrt{m}
=
\frac{1+\alpha}{2}x_c^\star-\frac{\alpha}{\sqrt{m}}+\alpha\rho\sqrt{m},
\]
which yields
\[
x_c^\star=\frac{2\alpha(1-\rho m)}{(1+\alpha)\sqrt{m}},
\]
and this is positive exactly when $\rho<1/m$. Now fix any leaf $\ell$. Since $x_\ell^\star=0$, the KKT condition requires
\(
\nabla_\ell f(x^\star)\in[-\alpha\rho\sqrt{d_\ell},\,0]=[-\alpha\rho,\,0].
\)
Here
\[
\nabla_\ell f(x^\star)=(Qx^\star)_\ell
=
Q_{\ell c}x_c^\star
=
-\frac{1-\alpha}{2\sqrt{m}}\,x_c^\star \le 0,
\]
so the inactive condition is equivalent to
\[
|\nabla_\ell f(x^\star)|
=
\frac{1-\alpha}{2\sqrt{m}}\,x_c^\star
\le \alpha\rho.
\]
Substituting the expression for $x_c^\star$ and cancelling $\alpha>0$ gives
\[
\frac{1-\alpha}{2\sqrt{m}}\cdot \frac{2\alpha(1-\rho m)}{(1+\alpha)\sqrt{m}}
\le \alpha\rho
\quad\Longleftrightarrow\quad
\frac{(1-\alpha)(1-\rho m)}{(1+\alpha)m}\le \rho
\quad\Longleftrightarrow\quad
\rho\ge \frac{1-\alpha}{2m}=\rho_0.
\]
Hence for $\rho\in[\rho_0,1/m)$, the point $x^\star=x_c^\star e_c$ satisfies all KKT conditions.
Since $F_\rho$ is $\alpha$-strongly convex, these KKT conditions certify that $x^\star$ is the unique minimizer and
$S^\star=\{c\}$. Finally, for any leaf $\ell$ (with $d_\ell=1$) the degree-normalized slack is
\[
\gamma_\ell
=
\frac{\lambda_\ell-|\nabla_\ell f(x^\star)|}{\sqrt{d_\ell}}
=
\alpha\rho-\frac{1-\alpha}{2\sqrt{m}}\,x_c^\star
=
\alpha\rho-\frac{1-\alpha}{2\sqrt{m}}\cdot \frac{2\alpha(1-\rho m)}{(1+\alpha)\sqrt{m}}
=
\frac{2\alpha}{1+\alpha}\,(\rho-\rho_0),
\]
so at $\rho=\rho_0$ we indeed have $\gamma=0$.
\end{proof}

\subsection{Path graph (seed at an endpoint)}

Let $G=P_{m+1}$ be the path on nodes $1,2,\dots,m+1$ with edges $(i,i+1)$.
Let $s=e_1$ (seed at endpoint $1$).
Assume $m\ge 2$, so that $d_1=d_{m+1}=1$ and $d_i=2$ for $2\le i\le m$.
Consider candidates of the form $x=x_1 e_1$.

\begin{lemma}[Path graph breakpoint: $\gamma$ can be $0$]\label{lem:path_gamma_small}
Fix $\alpha\in(0,1)$ and $m\ge 2$ and define
\[
\rho_0 \defi  \frac{1-\alpha}{3+\alpha}.
\]
For any $\rho\in[\rho_0,\,1)$, let
\[
x^\star \defi x_1^\star e_1,
\qquad
x_1^\star = \frac{2\alpha(1-\rho)}{1+\alpha}.
\]
Then $x^\star$ is a minimizer of $F_\rho$, hence the unique minimizer by $\alpha$-strong convexity. In particular, $S^\star=\{1\}$.
Moreover, the \emph{degree-normalized} KKT slack at node $2$ (where $d_2=2$), with
$\lambda_2=\alpha\rho\sqrt{d_2}=\alpha\rho\sqrt{2}$, equals
\[
\gamma_2
\;\defi\;
\frac{\lambda_2-|\nabla_2 f(x^\star)|}{\sqrt{d_2}}
\;=\;
\frac{\alpha(3+\alpha)}{2(1+\alpha)}\,(\rho-\rho_0).
\]
In particular, at the breakpoint $\rho=\rho_0$ one has $\gamma=0$.
\end{lemma}

\begin{proof}
Recall that $f(x)=\frac12 x^\top Qx-\alpha\langle D^{-1/2}s,x\rangle$, hence
\[
\nabla f(x)=Qx-\alpha D^{-1/2}s,
\qquad
Q=\frac{1+\alpha}{2}I-\frac{1-\alpha}{2}D^{-1/2}AD^{-1/2}.
\]
Since $s=e_1$ and $d_1=1$, we have $D^{-1/2}s=e_1$.
Also, for the edge $(1,2)$ we have $(D^{-1/2}AD^{-1/2})_{21}=1/\sqrt{d_2d_1}=1/\sqrt{2}$, so
\[
Q_{11}=\frac{1+\alpha}{2},
\qquad
Q_{21}=Q_{12}=-\frac{1-\alpha}{2\sqrt{2}}.
\]

Assume $x^\star=x_1^\star e_1$ with $x_1^\star>0$.
For the $\ell_1$-regularized PageRank objective $F_\rho(x)=f(x)+\alpha\rho\|D^{1/2}x\|_1$,
the active KKT condition at node $1$ is
\[
\nabla_1 f(x^\star)=-\alpha\rho\sqrt{d_1}=-\alpha\rho.
\]
But $\nabla_1 f(x^\star)=(Qx^\star)_1-\alpha=Q_{11}x_1^\star-\alpha$, hence
\[
0=\nabla_1 f(x^\star)+\alpha\rho
=
Q_{11}x_1^\star-\alpha+\alpha\rho
=
\frac{1+\alpha}{2}\,x_1^\star-\alpha+\alpha\rho,
\]
which yields
\[
x_1^\star=\frac{2\alpha(1-\rho)}{1+\alpha},
\]
and this is positive iff $\rho<1$. Now consider node $2$ (which is inactive under our candidate).
Since $(D^{-1/2}s)_2=0$ and $x^\star$ is supported only on node $1$,
\[
\nabla_2 f(x^\star)=(Qx^\star)_2
=
Q_{21}x_1^\star
=
-\frac{1-\alpha}{2\sqrt{2}}\,x_1^\star \le 0,
\]
so
\[
|\nabla_2 f(x^\star)|
=
\frac{1-\alpha}{2\sqrt{2}}\,x_1^\star.
\]
The inactive KKT condition at node $2$ requires
\(
|\nabla_2 f(x^\star)|\le \alpha\rho\sqrt{d_2}=\alpha\rho\sqrt{2}.
\)
Substituting $x_1^\star$ gives
\[
\frac{1-\alpha}{2\sqrt{2}}\cdot\frac{2\alpha(1-\rho)}{1+\alpha}
\le
\alpha\rho\sqrt{2}
\quad\Longleftrightarrow\quad
\frac{(1-\alpha)(1-\rho)}{1+\alpha}\le 2\rho
\quad\Longleftrightarrow\quad
\rho\ge \frac{1-\alpha}{3+\alpha}=\rho_0.
\]

For nodes $i\ge 3$, we have $(Qx^\star)_i=Q_{i1}x_1^\star=0$ because node $1$ is adjacent only to node $2$,
and also $(D^{-1/2}s)_i=0$, hence $\nabla_i f(x^\star)=0$, which satisfies the inactive KKT condition
$|\nabla_i f(x^\star)|\le \alpha\rho\sqrt{d_i}$.
Therefore, for any $\rho\in[\rho_0,1)$, the point $x^\star=x_1^\star e_1$ satisfies all KKT conditions.
Since $F_\rho$ is $\alpha$-strongly convex, this certifies that $x^\star$ is the unique minimizer and $S^\star=\{1\}$. Finally, the degree-normalized slack at node $2$ is
\begin{align*}
\gamma_2
&=
\frac{\lambda_2-|\nabla_2 f(x^\star)|}{\sqrt{d_2}}
=
\frac{\alpha\rho\sqrt{2}-\frac{1-\alpha}{2\sqrt{2}}x_1^\star}{\sqrt{2}}
=
\alpha\rho-\frac{1-\alpha}{4}x_1^\star\\
&=
\alpha\rho-\frac{1-\alpha}{4}\cdot\frac{2\alpha(1-\rho)}{1+\alpha}
=
\frac{\alpha}{2(1+\alpha)}\Bigl((3+\alpha)\rho-(1-\alpha)\Bigr)
=
\frac{\alpha(3+\alpha)}{2(1+\alpha)}\,(\rho-\rho_0),
\end{align*}
and at $\rho=\rho_0$ this slack is $0$, so $\gamma=0$.
\end{proof}

\section{FISTA can be worse than ISTA: a lower bound}
\label{sec:lower_bound}

We exhibit a family of star instances for which ISTA remains supported on the seed leaf and therefore has graph-size-independent work, whereas standard FISTA activates the high-degree center after two extrapolated steps, and incurs $\Omega(m)$ degree-weighted work, where $m+1$ is the number of nodes in the star graph. Consequently, for a fixed target accuracy depending only on $\alpha$, FISTA can be asymptotically worse than ISTA by a factor linear in the graph size. The proof is as follows: we identify the regularization level for which the center stays inactive at optimality, show that activation for FISTA reduces to a seed-coordinate condition, prove that FISTA satisfies the condition after two steps, and then show that the resulting cost is incurred before the target accuracy is reached.

\subsection{Construction}

Fix an integer $m\ge 2$.
Let $G(m)$ be the star graph on $n=m+1$ vertices with vertex set $\{w,v,u_1,\dots,u_{m-1}\}$, where $w$ is the center and $v,u_1,\dots,u_{m-1}$ are the leaves.
The edge set is $\bigl\{\{w,v\}\bigr\}\cup\bigl\{\{w,u_i\}:i=1,\dots,m-1\bigr\}$.
Thus every leaf is adjacent only to the center $w$, and there are no edges between distinct leaves.
In particular, the center has degree $d_w=m$, while each leaf has degree $1$, that is, $d_v=d_{u_i}=1$ for all $i=1,\dots,m-1$.
We distinguish $v$ as the seed node.
For the regularization regime in this section, the optimal solution is supported only on the seed leaf, so $S^\star=\{v\}$, $\bdry S^\star=\{w\}$, and $\mathrm{Ext}(S^\star)=\{u_1,\dots,u_{m-1}\}$.
Hence $\vol(S^\star)=1$ and $\vol(\bdry S^\star)=m$.

\subsection{Results}

The following lemma pins down the optimal solution and the critical regularization breakpoint.

\begin{lemma}\label{lem:star_leaf_breakpoint}
For the graph $G(m)$ with seed $s=e_v$ and any
$\rho\in[\rho_0,1)$ where
\[
\rho_0 \;\defi\; \frac{1-\alpha}{m(1+\alpha)+(1-\alpha)},
\]
let
\[
x^\star \defi x_v^\star e_v,
\qquad
x_v^\star = \frac{2\alpha(1-\rho)}{1+\alpha}.
\]
Then $x^\star$ is a minimizer of $F_\rho$ (cf.\ \cref{eq:reg_personalized_pagerank}), hence the unique minimizer by $\alpha$-strong convexity. In particular, $S^\star=\{v\}$. The degree-normalized complementarity margin at the center $w$ is
\[
\gamma_w
\;=\;
\frac{\alpha\bigl(m(1+\alpha)+(1-\alpha)\bigr)}{(1+\alpha)m}\,(\rho - \rho_0).
\]
In particular, at $\rho=\rho_0$ the margin is $\gamma_w=0$.
\end{lemma}

\begin{proof}
Since $d_v=1$ and $s=e_v$, we have $D^{-1/2}s=e_v$.
The PageRank matrix satisfies
\[
Q_{vv}=\frac{1+\alpha}{2},
\qquad
Q_{wv}=Q_{vw}=-\frac{1-\alpha}{2\sqrt{m}},
\qquad
Q_{u_i v}=0 \quad \text{for all } i,
\]
because there are no edges between leaves.
Assume $x^\star = x_v^\star e_v$ with $x_v^\star > 0$.
The active KKT condition at $v$ gives
\[
Q_{vv}x_v^\star - \alpha = -\alpha\rho,
\]
so
\[
x_v^\star = \frac{2\alpha(1-\rho)}{1+\alpha},
\]
which is positive for $\rho<1$. For the center $w$ (with $d_w=m$),
\[
\nabla_w f(x^\star) = Q_{wv}x_v^\star
= -\frac{1-\alpha}{2\sqrt{m}}\cdot\frac{2\alpha(1-\rho)}{1+\alpha}
= -\frac{\alpha(1-\alpha)(1-\rho)}{(1+\alpha)\sqrt{m}}.
\]
The inactive KKT condition at $w$ requires
$|\nabla_w f(x^\star)|\le \alpha\rho\sqrt{m}$, i.e.,
\[
\frac{(1-\alpha)(1-\rho)}{1+\alpha}\le \rho m.
\]
Solving the equality case yields
\[
\rho_0 = \frac{1-\alpha}{m(1+\alpha)+(1-\alpha)},
\]
and thus the condition holds for all $\rho\ge \rho_0$. Each pendant leaf $u_i$ is neither the seed nor adjacent to $v$, so
\[
|\nabla_{u_i} f(x^\star)|=0\le \alpha\rho,
\]
and its KKT condition is satisfied.
By $\alpha$-strong convexity, $x^\star$ is the unique minimizer and $S^\star=\{v\}$. Finally, the degree-normalized margin at $w$ is
\[
\gamma_w
=
\alpha\rho - \frac{|\nabla_w f(x^\star)|}{\sqrt{m}}
=
\alpha\rho - \frac{\alpha(1-\alpha)(1-\rho)}{(1+\alpha)m}
=
\frac{\alpha\bigl(m(1+\alpha)+(1-\alpha)\bigr)}{(1+\alpha)m}\,(\rho-\rho_0),
\]
which vanishes exactly at $\rho=\rho_0$.
\end{proof}

We next derive the exact criterion for activation of the center $w$.
The FISTA update decides activation through the forward point $u(y)=y-\nabla f(y)$, since the proximal step makes the coordinate $w$ nonzero exactly when $|u(y)_w|$ exceeds the weighted soft-threshold $\alpha\rho_0\sqrt{m}$.
On the star graph, if $y$ is supported on $\{v,w\}$, then a direct calculation shows that $u(y)_w=\frac{1-\alpha}{2}(y_w+y_v/\sqrt{m})$, so the center is influenced only by its own value and by the seed value transmitted through the unique edge $(v,w)$.
The breakpoint $\rho=\rho_0$ is chosen so that, at the optimum supported only on the seed leaf, the center $w$ is exactly at the point where the proximal update changes from keeping it zero to making it nonzero, namely $\frac{1-\alpha}{2\sqrt{m}}x_v^\star=\alpha\rho_0\sqrt{m}$.
Rewriting the threshold test using this identity yields the criterion below.
In particular, before the first activation, when $y_w=0$ and the iterates are nonnegative, the condition reduces to $y_v>x_v^\star$.
Thus the lemma is the bridge to \Cref{lem:fista_overshoot}: once we prove that FISTA overshoots the seed coordinate beyond $x_v^\star$, activation of the center follows immediately.

\begin{lemma}\label{lem:activation_criterion}
At $\rho=\rho_0$, consider any point $y$ with $\supp(y)\subseteq\{v,w\}$.
Let
\[
x^+ \defi \prox_{\alpha\rho_0\|D^{1/2}\cdot\|_1}\!\bigl(u(y)\bigr),
\qquad
u(y)\defi y-\nabla f(y).
\]
Then
\[
x^+_{w}\neq 0
\qquad\Longleftrightarrow\qquad
|u(y)_w|> \alpha\rho_0\sqrt{m}
\qquad\Longleftrightarrow\qquad
|y_{w}\sqrt{m}+y_{v}|>x_v^\star.
\]
In particular, if $y_w,y_v\ge 0$, then
\[
x^+_{w}>0
\qquad\Longleftrightarrow\qquad
y_{w}\sqrt{m}+y_{v}>x_v^\star.
\]
\end{lemma}

\begin{proof}
By the coordinate formula for the proximal operator,
\[
x^+_{w}\neq 0
\qquad\Longleftrightarrow\qquad
|u(y)_w| > \alpha\rho_0\sqrt{m}.
\]
Since $\supp(y)\subseteq\{v,w\}$ and $w$ is not the seed,
\[
u(y)_w
=
\bigl((I-Q)y\bigr)_w
=
\frac{1-\alpha}{2}\left(y_w+\frac{y_v}{\sqrt{m}}\right).
\]
At the breakpoint,
\[
\frac{1-\alpha}{2\sqrt{m}}\,x_v^\star
=
\alpha\rho_0\sqrt{m}.
\]
Therefore,
\[
|u(y)_w| > \alpha\rho_0\sqrt{m}
\quad\Longleftrightarrow\quad
\frac{1-\alpha}{2}\left|y_w+\frac{y_v}{\sqrt{m}}\right|
>
\frac{1-\alpha}{2\sqrt{m}}x_v^\star
\quad\Longleftrightarrow\quad
|y_w\sqrt{m}+y_v|>x_v^\star.
\]
If $y_w,y_v\ge0$, then $u(y)_w\ge0$, so $x_w^+>0$ iff $x_w^+\neq0$, giving the last claim.
\end{proof}

We now show that FISTA generates exactly the condition required by \Cref{lem:activation_criterion}.
The key point is that, before the center becomes active, every iterate remains supported on the seed leaf $v$, so the dynamics reduce to a one-dimensional accelerated proximal-gradient iteration on the seed coordinate alone.
In this regime, the update at $v$ is affine, and the error relative to the optimum, $e_k \defi x_{k,v}-x_v^\star$, satisfies an explicit scalar recurrence.
This allows us to compute the first few iterates exactly.
We verify first that the extrapolated points at $k=0$ and $k=1$ do not cross the activation threshold, so the center is still inactive.
At $k=2$, however, the momentum term pushes the extrapolated seed coordinate past the critical value $x_v^\star$, that is, $y_{2,v}>x_v^\star$.
By \Cref{lem:activation_criterion}, this activates the center $w$.

\begin{lemma}\label{lem:fista_overshoot}
Run \cref{eq:apg_fista} on $F_{\rho_0}$ (cf.~\cref{eq:reg_personalized_pagerank})
for the graph $G(m)$ with seed $s=e_v$, starting from $x_{-1}=x_0=0$.
Then
\begin{equation}\label{eq:overshoot_exact}
y_{2,v} - x_v^\star
= \frac{(1-\alpha)\beta^2}{2}\, x_v^\star
> 0.
\end{equation}
Consequently, FISTA activates the center $w$ at iteration $k=2$.
\end{lemma}

\begin{proof}
At $k=0$, we have $y_0=0$, so only the seed coordinate can become active.
Thus
\[
x_{1,v}=\alpha(1-\rho_0)>0,
\qquad
x_{1,w}=x_{1,u_i}=0.
\]
Hence $x_1$ is supported on $\{v\}$. Define the errors
\[
e_k \defi x_{k,v}-x_v^\star,
\qquad
\tilde e_k \defi y_{k,v}-x_v^\star.
\]
As long as $w$ has not been activated, both $x_k$ and $y_k$ are supported on $\{v\}$.
On such steps,
\[
u(y_k)_v=(1-Q_{vv})y_{k,v}+\alpha = \frac{1-\alpha}{2}y_{k,v}+\alpha.
\]
Since $y_{k,v}\ge 0$ on the steps we consider, the soft-threshold at $v$ acts affinely, and therefore
\[
x_{k+1,v}
=
u(y_k)_v-\alpha\rho_0
=
\frac{1-\alpha}{2}y_{k,v}+\alpha(1-\rho_0).
\]
Writing
\[
a \defi \frac{1-\alpha}{2},
\]
and subtracting the fixed-point identity
\[
x_v^\star = a\,x_v^\star+\alpha(1-\rho_0),
\]
we get
\[
e_{k+1}=a\,\tilde e_k
= a\bigl((1+\beta)e_k-\beta e_{k-1}\bigr).
\]
Equivalently,
\begin{equation}\label{eq:1d_recurrence}
e_{k+1}=a(1+\beta)e_k-a\beta e_{k-1}.
\end{equation}

\emph{Initial values.}
We have
\[
e_0 = -x_v^\star.
\]
The first FISTA step gives
\[
x_{1,v}=\alpha(1-\rho_0),
\]
hence
\[
e_1
=
\alpha(1-\rho_0)-\frac{2\alpha(1-\rho_0)}{1+\alpha}
=
\frac{1-\alpha}{2}\,e_0
=
a e_0.
\]

\emph{The center is not activated at $k=0$ or $k=1$.}
At $k=0$, $y_{0,v}=0<x_v^\star$, so \Cref{lem:activation_criterion} shows that $w$ is not activated. At $k=1$, since $x_{1,w}=x_{0,w}=0$, we have $y_{1,w}=0$, and
\[
\tilde e_1=(1+\beta)e_1-\beta e_0
= e_0\bigl((1+\beta)a-\beta\bigr).
\]
Using
\[
(1+\beta)a = 1-\sqrt{\alpha}
\qquad\text{and}\qquad
(1+\beta)a-\beta = \sqrt{\alpha}\beta>0,
\]
we get
\[
\tilde e_1=\sqrt{\alpha}\beta\,e_0<0
\]
because $e_0<0$.
Thus $y_{1,v}<x_v^\star$, and \Cref{lem:activation_criterion} again implies that $w$ is not activated at $k=1$.

\emph{Computing $\tilde e_2$.}
Since $w$ is not activated at $k=0$ or $k=1$, the recurrence \eqref{eq:1d_recurrence} applies up to $e_2$:
\[
e_2
=
a(1+\beta)e_1-a\beta e_0
=
a e_0\bigl(a(1+\beta)-\beta\bigr)
=
a\sqrt{\alpha}\beta\,e_0.
\]
Therefore,
\begin{align*}
\tilde e_2
&=
(1+\beta)e_2-\beta e_1 \\
&=
(1+\beta)\,a\sqrt{\alpha}\beta\,e_0-\beta a e_0 \\
&=
a\beta e_0\bigl((1+\beta)\sqrt{\alpha}-1\bigr).
\end{align*}
Now
\[
(1+\beta)\sqrt{\alpha}
=
\frac{2\sqrt{\alpha}}{1+\sqrt{\alpha}},
\]
so
\[
(1+\beta)\sqrt{\alpha}-1
=
\frac{\sqrt{\alpha}-1}{1+\sqrt{\alpha}}
=
-\beta.
\]
Hence
\[
\tilde e_2
=
-a\beta^2 e_0
=
a\beta^2 x_v^\star
=
\frac{(1-\alpha)\beta^2}{2}\,x_v^\star
>
0.
\]
Thus $y_{2,v}>x_v^\star$.
Also, since $x_{2,w}=x_{1,w}=0$, we have $y_{2,w}=0$.
Applying \Cref{lem:activation_criterion} with $y=y_2$ therefore shows that FISTA activates $w$ at iteration $k=2$.
\end{proof}

We now convert the activation of the center into a lower bound on the total degree-weighted work.
Once \Cref{lem:fista_overshoot} shows that the center becomes active at iteration $k=2$, the next iterate $x_3$ already contains the high-degree node $w$, and the following extrapolated point $y_3$ contains it as well.
Under our work model, this immediately creates work of order $m$ in two successive iterations.
To obtain a lower bound for reaching a prescribed accuracy, it remains to show that this expensive activation occurs before FISTA can terminate.
We therefore bound the objective gap explicitly along the first few iterates and prove that, for every target $\varepsilon\le \varepsilon_0(\alpha)$, none of $x_0,x_1,x_2,x_3$ is yet $\varepsilon$-accurate.
Hence any successful run must execute at least four iterations and must incur at least $2m$ total work.
By contrast, ISTA remains supported on the seed leaf throughout, so its work stays independent of $m$.

\begin{proposition}\label{prop:lower_bound}
Fix $\alpha\in(0,1)$ and define
\[
\varepsilon_0(\alpha)
\;\defi\;
\frac{\alpha^3(1-\alpha)^4\beta^4}{2(3+\alpha)^2}
\;>\;0,
\qquad
\beta=\frac{1-\sqrt{\alpha}}{1+\sqrt{\alpha}}.
\]
On the graph $G(m)$ with seed $s=e_v$ and $\rho=\rho_0$, for every target accuracy
\[
0<\varepsilon\le \varepsilon_0(\alpha),
\]
standard FISTA requires total degree-weighted work at least $2m$ to reach
\[
F_{\rho_0}(x_N)-F_{\rho_0}(x^\star)\le \varepsilon.
\]
By contrast, ISTA reaches the same target with total work
\[
O\!\left(\frac{1}{\alpha} \log\!\frac{1}{\varepsilon}\right),
\]
independent of $m$.
\end{proposition}

\begin{proof}
Let
\[
a \defi \frac{1-\alpha}{2}.
\]

\emph{FISTA lower bound.}
By \Cref{lem:fista_overshoot}, FISTA activates $w$ at iteration $k=2$, i.e.,
$x_{3,w}>0$.
Hence
\[
w\in \supp(x_3)
\qquad\text{and}\qquad
\vol(\supp(x_3))\ge \deg(w)=m.
\]
Also, $x_{2,w}=0$, so
\[
y_3 = x_3+\beta(x_3-x_2)
\]
satisfies
\[
y_{3,w}=(1+\beta)x_{3,w}>0.
\]
Therefore
\[
w\in \supp(y_3)
\qquad\text{and}\qquad
\vol(\supp(y_3))\ge m.
\]
Thus iterations $k=2$ and $k=3$ each incur per-iteration work at least $m$:
\[
\mathrm{work}_2\ge \vol(\supp(x_3))\ge m,
\qquad
\mathrm{work}_3\ge \vol(\supp(y_3))\ge m.
\]
It remains to show that, for every target accuracy $0<\varepsilon\le \varepsilon_0(\alpha)$, the algorithm must execute at least four iterations. For $k=0,1,2$, the center has not yet been activated, so $x_k$ is supported on $\{v\}$.
Moreover these iterates are nonnegative.
Hence, on the ray $\{x e_v:\; x\ge 0\}$,
\[
F_{\rho_0}(x e_v)
=
\frac{1+\alpha}{4}x^2-\alpha(1-\rho_0)x,
\]
and therefore
\[
F_{\rho_0}(x_k)-F_{\rho_0}(x^\star)
=
\frac{1+\alpha}{4}(x_{k,v}-x_v^\star)^2.
\]
From the proof of \Cref{lem:fista_overshoot}, the corresponding errors satisfy
\[
e_0 \defi x_{0,v}-x_v^\star = -x_v^\star,\qquad
e_1 = -a x_v^\star,\qquad
e_2 = -a\sqrt{\alpha}\beta\,x_v^\star.
\]
Since $a,\sqrt{\alpha}\beta\in(0,1)$, the smallest of the first three gaps occurs at $k=2$, and therefore
\[
F_{\rho_0}(x_k)-F_{\rho_0}(x^\star)
\;\ge\;
F_{\rho_0}(x_2)-F_{\rho_0}(x^\star)
=
\frac{1+\alpha}{4}a^2\alpha\beta^2 (x_v^\star)^2
\qquad\text{for }k=0,1,2.
\]
At $\rho=\rho_0$,
\[
x_v^\star
=
\frac{2\alpha(1-\rho_0)}{1+\alpha}
=
\frac{2\alpha m}{m(1+\alpha)+(1-\alpha)}
\ge \frac{4\alpha}{3+\alpha},
\]
where the last inequality uses $m\ge2$.
Substituting this bound yields
\[
F_{\rho_0}(x_k)-F_{\rho_0}(x^\star)
\ge
\frac{1+\alpha}{4}a^2\alpha\beta^2
\left(\frac{4\alpha}{3+\alpha}\right)^2
=
\frac{\alpha^3(1+\alpha)(1-\alpha)^2\beta^2}{(3+\alpha)^2}
>
\varepsilon_0(\alpha)
\]
for $k=0,1,2$, because
\[
2(1+\alpha)>(1-\alpha)^2\beta^2.
\]
For $k=3$, using \Cref{lem:fista_overshoot} and the $v$-update,
\[
x_{3,v}-x_v^\star
=
a(y_{2,v}-x_v^\star)
=
a^2\beta^2 x_v^\star.
\]
By $\alpha$-strong convexity,
\[
F_{\rho_0}(x_3)-F_{\rho_0}(x^\star)
\ge
\frac{\alpha}{2}\|x_3-x^\star\|_2^2
>
\frac{\alpha}{2}(x_{3,v}-x_v^\star)^2,
\]
where the inequality is strict because $x_{3,w}>0$ while $x_w^\star=0$.
Using the bound on $x_v^\star$ above,
\[
F_{\rho_0}(x_3)-F_{\rho_0}(x^\star)
>
\frac{\alpha}{2}\left(a^2\beta^2\cdot \frac{4\alpha}{3+\alpha}\right)^2
=
\varepsilon_0(\alpha).
\]
Hence
\[
F_{\rho_0}(x_N)-F_{\rho_0}(x^\star)>\varepsilon_0(\alpha)\ge \varepsilon
\qquad\text{for every }N\le 3.
\]
So any run that reaches
\[
F_{\rho_0}(x_N)-F_{\rho_0}(x^\star)\le \varepsilon
\qquad\text{with }0<\varepsilon\le \varepsilon_0(\alpha)
\]
must have $N\ge4$.
Since total work sums over iterations,
\[
\mathrm{Work}(N)\ge \mathrm{work}_2+\mathrm{work}_3\ge 2m.
\]

\emph{ISTA upper bound on the same instance.} By Theorem 1(ii) in \cite{fountoulakis2019variational}, the support of each ISTA iterate is contained in the optimal support. Furthermore, \cref{lem:star_leaf_breakpoint} shows that the optimal support is \(S^\star=\{v\}\), so \(|S^\star|=1\). Hence, the per-iteration work of ISTA is \(\mathcal{O}(1)\) with respect to \(m\). Moreover, Theorem 10.30 in \cite{beck2017first} states that ISTA requires \(\mathcal{O}\!\left(\frac{1}{\alpha}\log\frac{1}{\epsilon}\right)\) iterations to obtain a solution whose objective value is within \(\epsilon\) of the optimum. Therefore, the total work of ISTA is \(\mathcal{O}\!\left(\frac{1}{\alpha}\log\frac{1}{\epsilon}\right)\), which is independent of \(m\).
\end{proof}

\section{Experimental setting details}
\label{app:exp_setting_details}

This section collects the common experimental ingredients used throughout the synthetic experiments
in \Cref{subsec:synthetic_volB_sweep,subsec:b600_sweeps}.
All experiments solve the $\ell_1$-regularized PageRank objective \Cref{eq:reg_personalized_pagerank} and report
runtime using the degree-weighted work metric in \Cref{eq:deg_work_model}.
When we refer to the no-percolation diagnostic, we mean the inequality from \Cref{thm:boundary_confinement_exposure}.

\textbf{Synthetic graph family: core-boundary-exterior construction.}
Each synthetic instance is an undirected graph with a three-way partition of the node set $V =S \;\cup\; \mathcal{B} \;\cup\; \mathrm{Ext}$,
where $S$ is a core (containing the seed), $\mathcal{B}$ is a boundary region, and
$\mathrm{Ext}$ is an exterior. The construction is deterministic. Given sizes $|S|$, $|\mathcal{B}|$, and
$|\mathrm{Ext}|$, edges are added according to the following rules:
\begin{itemize}
    \item \textit{Core clique.} The induced subgraph on $S$ is a complete graph (a clique).

    \item \textit{Core-boundary connectivity.}
    Let the core nodes be ordered as $S=\{0,1,\dots,|S|-1\}$ and let the boundary nodes be stored in an
    ordered list $(b_0,b_1,\dots,b_{|\mathcal{B}|-1})$.
    Each core node has
    $c_{\mathrm{bnd}}$ boundary per core neighbors in $\mathcal{B}$. For each core node $u\in S$ and each $j\in\{0,1,\dots,c_{\mathrm{bnd}}-1\}$ we add the edge $(u,\; b_{ (u\cdot c_{\mathrm{bnd}}+j)\bmod |\mathcal{B}| })$.
    When $|\mathcal{B}|\ge c_{\mathrm{bnd}}$ (as in our sweeps), this gives $c_{\mathrm{bnd}}$ distinct
    boundary neighbors per core node. Each core node has fixed degree
    \(
    d_u = (|S|-1)+c_{\mathrm{bnd}}
    \)
    for $|\mathcal{B}|>0$.

    \item \textit{Boundary internal connectivity.} The boundary induces a circulant graph with an even degree
    parameter $\deg_{\mathcal{B}}$, capped at $|\mathcal{B}|-1$, and adjusted to be even.
    \item \textit{Exterior internal connectivity.} The exterior induces a circulant graph with degree
    $\deg_{\mathrm{Ext}}$, with $\deg_{\mathrm{Ext}}<|\mathrm{Ext}|$.
    \item \textit{Boundary-exterior connectivity.} Each exterior node has exactly one neighbor in $\mathcal{B}$, using the same rule as above, so the number of boundary-exterior edges equals $|\mathrm{Ext}|$.
\end{itemize}
This construction yields a dense core, an internally connected boundary band, and a highly connected
exterior, with sparse cross-region interfaces. When we visualize adjacency matrices, this produces
a clear block structure (core | boundary | exterior) and a boundary region whose size/volume can be
varied independently of the core neighborhood.

\textbf{Optimization objective and parameters.}
On each graph instance we solve the $\ell_1$-regularized PageRank objective
\Cref{eq:reg_personalized_pagerank} with a single-node seed $s=e_v$.
Unless otherwise specified, the seed node $v$ is a fixed core vertex (in the code, $v=0$). Each experiment specifies a teleportation parameter $\alpha\in(0,1]$ and a sparsity parameter
$\rho>0$. When using FISTA we set the momentum parameter to the standard strongly-convex choice $\beta \;\defi\; \frac{1-\sqrt{\alpha}}{1+\sqrt{\alpha}}$
(for PageRank, $L=1$ and $\mu=\alpha$). Both ISTA and FISTA are initialized at $x_{-1}=x_0=0$.

\textbf{Stopping criterion.}
All experiments compare ISTA and FISTA under the same KKT surrogate based on the proximal-gradient fixed point.
With unit step size, define the prox-gradient map
\[
T_{\alpha,\rho}(x)\;\defi\;\prox_{g}\!\bigl(x-\nabla f(x)\bigr),
\qquad
r(x)\;\defi\;\|x-T_{\alpha,\rho}(x)\|_{\infty}.
\]
A point $x^\star$ is optimal for \Cref{eq:reg_personalized_pagerank} if and only if $x^\star=T_{\alpha,\rho}(x^\star)$, i.e., $r(x^\star)=0$.
We therefore declare convergence when the fixed-point residual satisfies $r(x_k)\le \varepsilon$,
where $\varepsilon>0$ is the prescribed tolerance.
This termination rule is applied identically to ISTA and FISTA.
In the work-vs-$\varepsilon$ sweeps, the $x$-axis parameter \emph{is} this residual tolerance $\varepsilon$; for the other sweeps, $\varepsilon$ is held fixed (and we impose a single large global iteration cap, e.g.\ $50{,}000$, for all runs). We terminate the algorithm based on the residual rather than the objective value, since computing it does not require knowing the optimal solution.

\textbf{Degree-weighted work model.}
We measure runtime via a degree-weighted work model \Cref{eq:deg_work_model}.
For an iterate pair $(y_k,x_{k+1})$ we define the per-iteration work as
$
\mathrm{work}_k
\;\defi\;
\vol(\supp(y_k)) + \vol(\supp(x_{k+1}))
$.
For ISTA, $y_k=x_k$; for FISTA, $y_k=x_k+\beta(x_k-x_{k-1})$.
The work to reach the stopping target is the sum of $\mathrm{work}_k$ over the iterations
taken.

\textbf{No-percolation diagnostic.}
The no-percolation assumption \Cref{eq:exposure_assumption}
is satisfied for all our synthetic experiments. Conceptually, this condition is favorable for accelerated methods:
it rules out ``percolation'' of extrapolated iterates into the exterior, so FISTA is not penalized by activating a large,
highly connected ambient region. Nevertheless, our sweeps still exhibit regimes where FISTA does not improve work (and can be
slower than ISTA), showing that even when exterior exploration is provably suppressed, acceleration can lose due to transient
boundary activations.

\textbf{Default synthetic parameters.}
Unless a sweep varies them, the synthetic experiments use the baseline block sizes and degrees $|S|=60$, $|\mathrm{Ext}|=1000$, $c_{\mathrm{bnd}}=20$, $\deg_{\mathcal{B}}=82$, $\deg_{\mathrm{Ext}}=998$,
and a fixed seed $v\in S$ (node $0$ in the implementation). The specific sweep parameter(s) are described in the
corresponding experiment sections.

\textbf{Per-point graph generation and how to read sweep plots.}
Our theory gives instance-wise guarantees (each bound applies to every graph in the family), and the synthetic family itself is specified by coarse structural parameters (block sizes and target degrees), not a single fixed adjacency matrix.
Accordingly, in several sweeps we intentionally regenerate the synthetic instance at each $x$-axis value. In these cases, each dot should be interpreted as one representative draw from the family at that parameter value, i.e., a snapshot of what can happen empirically under the same coarse structure. This design avoids conclusions that are artifacts of one particular synthetic realization and is aligned with the worst-case nature of the theory.

\section{Full details for the fixed-boundary sweeps experiments}
\label{app:b600_sweeps_full}
We provide full details for the experiments in \cref{subsec:b600_sweeps}.
We follow the synthetic construction, algorithmic choices, and work-metric conventions
from \Cref{app:exp_setting_details}, and fix the boundary size to $|\mathcal{B}|=600$.
We sweep $\rho$ (with fresh graphs per point), and we additionally sweep $\alpha$ and the fixed-point residual tolerance $\varepsilon$ with $\rho=10^{-4}$ fixed (and all other baseline parameters fixed).

This experiment complements the boundary-volume sweep of \Cref{subsec:synthetic_volB_sweep} by holding the boundary size fixed ($|\mathcal{B}|=600$) and varying only the regularization strength $\rho$. The aim is to isolate the $\rho$-dependence suggested by \Cref{thm:double_reg_work} (both terms scale as $1/\rho$ when $\alpha$ and the boundary are fixed), and to check whether ISTA and FISTA respond similarly as $\rho$ increases, since their worst-case theoretical running time depends on $\rho$ in the same way. We run two versions of the $\rho$-sweep, both using a randomized graph per $\rho$:
\begin{itemize}
\item \textit{Dense-core sweep.} The core subgraph is a clique, see \Cref{app:exp_setting_details}. 
\item \textit{Sparse-core sweep.} The core subgraph is sparsified by retaining a fixed fraction of its edges while enforcing that the core remains connected (implemented by sampling a random spanning tree and then adding random core-core edges up to the target density). In the sparse variant used here we keep $20\%$ of the clique edges. We perform experiments on the sparsified-core variant to verify that the observed $\rho$-dependence is not an artifact of the highly symmetric clique core: sparsifying reduces and heterogenizes core/seed degrees. For both variants, we sweep $\rho$ over a log-spaced grid chosen so that the no-percolation inequality holds for all sampled values.
\end{itemize}

The next experiment sweeps $\alpha$, while keeping all other parameters fixed.
Sweeping $\alpha$ to smaller values makes the no-percolation condition more stringent.
Rather than reweighting edges, we keep the graph unweighted and use an $\alpha$-sweep-specific graph family in which the exterior is a complete graph on $|\mathrm{Ext}|$ nodes,
and only a prescribed number $m$ of exterior nodes have a single boundary neighbor (the remaining exterior nodes have no boundary neighbor).
We choose $|\mathrm{Ext}|$ so that the no-percolation inequality holds at the smallest swept value $\alpha_{\min}$; since the left-hand side decreases with $\alpha$, this implies no-percolation for all $\alpha\ge \alpha_{\min}$ in the sweep.

The $\alpha$ sweep in our code additionally includes an auto-tuning step that selects a single unweighted instance from this family before running the sweep.
Concretely, the tuner searches over: (i) the core-boundary fanout $c_{\mathrm{bnd}}$ (boundary neighbors per core node), (ii) the boundary internal circulant degree, and (iii) the number $m$ of exterior-to-boundary edges (one boundary neighbor for each of the first $m$ exterior nodes), with $|\mathrm{Ext}|$ set to the smallest value that enforces no-percolation at $\alpha_{\min}$.
For each candidate, it evaluates performance on a calibration grid of $12$ log-spaced $\alpha$ values in $[\alpha_{\min},0.9]$ and chooses the candidate that maximizes the fraction of calibration points where FISTA incurs larger work than ISTA.
This is meant to illustrate that acceleration can be counterproductive on some valid instances even when iteration complexity improves.

For the $\varepsilon$ sweep, we keep $\alpha=0.20$ fixed and vary the fixed-point residual tolerance over a log-spaced grid $\varepsilon\in[10^{-12},10^{-1}]$.
We use the original baseline instance (no auto-tuning and no graph modification).

\section{Additional real-data diagnostics}
\label{app:real_diagnostics}

In this section we interpret the results of the experiments on real data from \cref{subsec:real_data_experiments}.

\textbf{Diagnosing slowdowns: iterations vs.\ per-iteration work.}
The work metric counts degree-weighted support volumes touched by both the extrapolated point $y_k$ and the proximal update,
so FISTA can lose either by taking more iterations than ISTA or by having a larger per-iteration locality cost.
To separate these effects, for each seed (at $\alpha=10^{-3}$, $\rho=10^{-4}$, $\varepsilon=10^{-8}$) we plot
\[
\text{iteration ratio} \;=\; \frac{N_{\mathrm{F}}}{N_{\mathrm{I}}}
\qquad\text{vs.}\qquad
\text{per-iter ratio} \;=\; \frac{(W_{\mathrm{F}}/N_{\mathrm{F}})}{(W_{\mathrm{I}}/N_{\mathrm{I}})},
\]
where $N_{\mathrm{I}},N_{\mathrm{F}}$ are iteration counts and $W_{\mathrm{I}},W_{\mathrm{F}}$ are total works.
Since $\frac{W_{\mathrm{F}}}{W_{\mathrm{I}}} = \frac{N_{\mathrm{F}}}{N_{\mathrm{I}}}\cdot
\frac{(W_{\mathrm{F}}/N_{\mathrm{F}})}{(W_{\mathrm{I}}/N_{\mathrm{I}})}$, points with both ratios above $1$
correspond to clear slowdowns.
\Cref{fig:real_tradeoff_iters_cost} shows that on \texttt{com-Orkut} at $\alpha=10^{-3}$, FISTA is frequently slower because it
often incurs both a larger iteration count and a larger per-iteration work cost, whereas on the other datasets FISTA typically
reduces iterations while paying a moderate per-iteration locality overhead.

\begin{figure}[h!]
    \centering

    \begin{subfigure}[t]{0.48\linewidth}
        \centering
        \includegraphics[width=\linewidth]{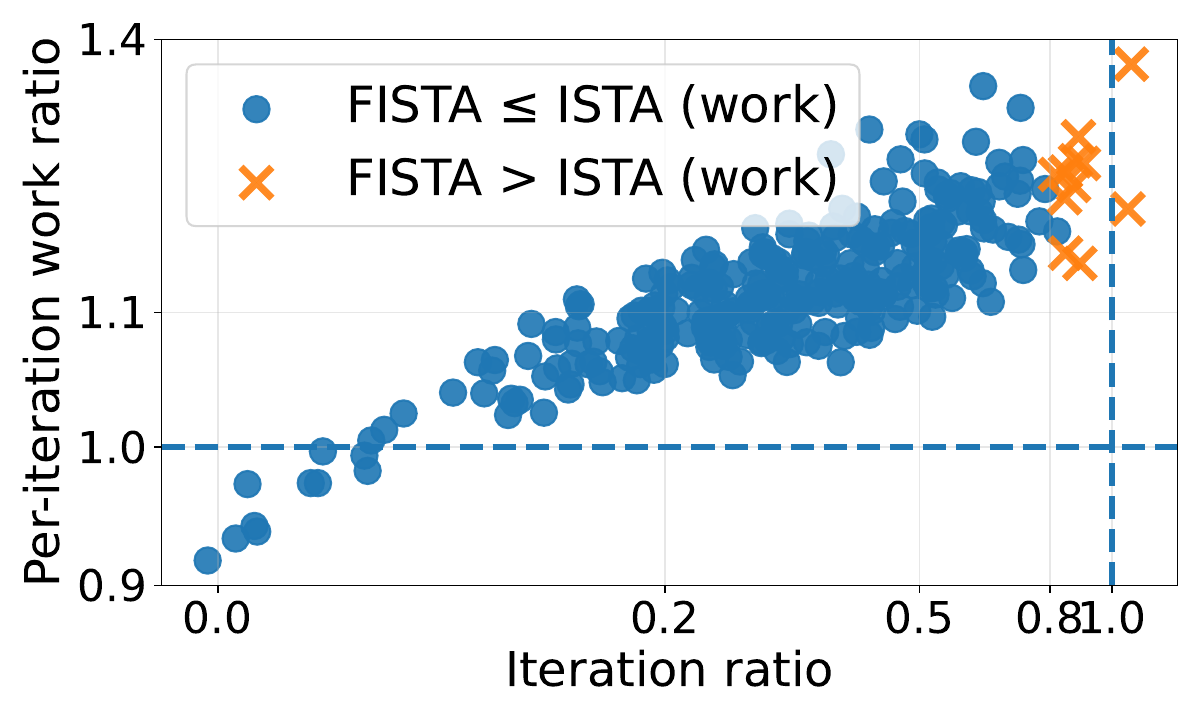}
        \caption{\texttt{com-Amazon}.}
        \label{fig:real_tradeoff_amazon}
    \end{subfigure}\hfill
    \begin{subfigure}[t]{0.48\linewidth}
        \centering
        \includegraphics[width=\linewidth]{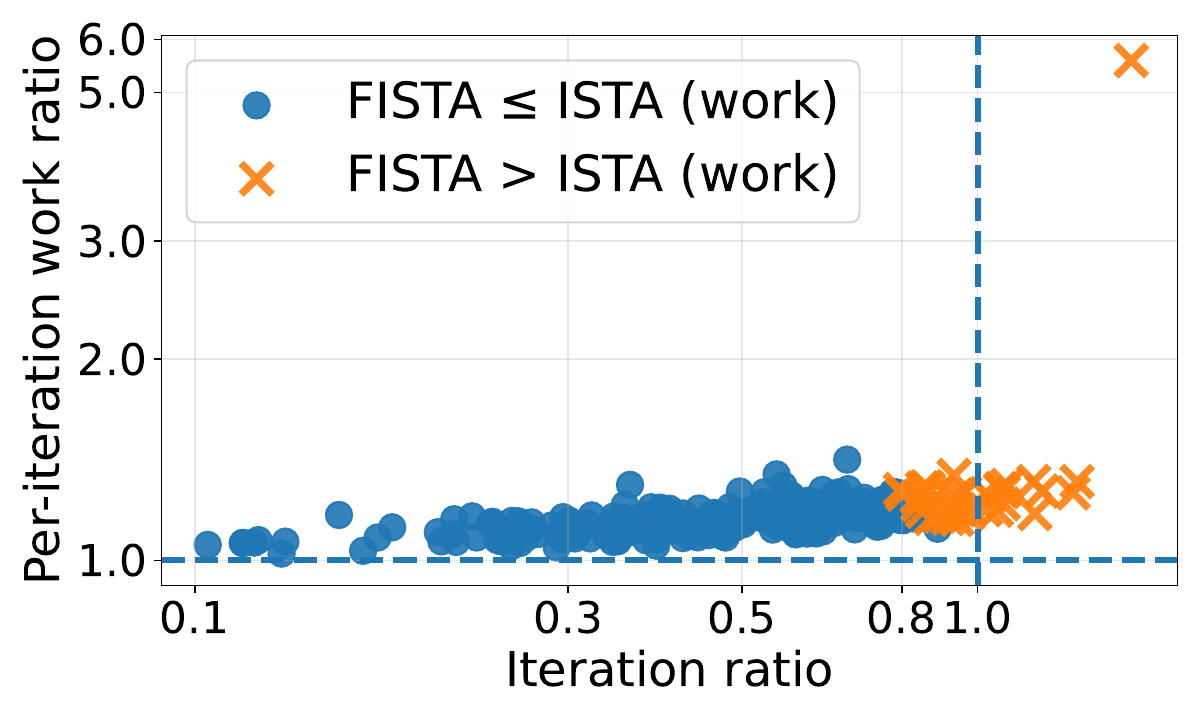}
        \caption{\texttt{com-DBLP}.}
        \label{fig:real_tradeoff_dblp}
    \end{subfigure}

    \vspace{0.6em}

    \begin{subfigure}[t]{0.48\linewidth}
        \centering
        \includegraphics[width=\linewidth]{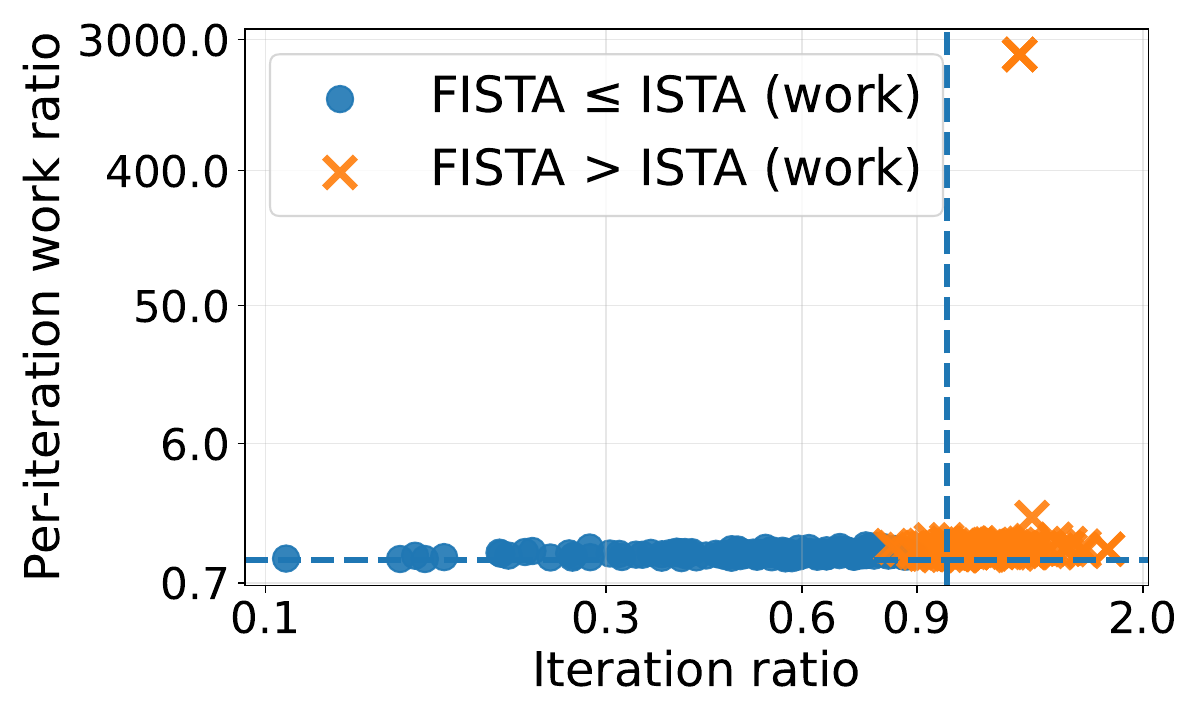}
        \caption{\texttt{com-Youtube}.}
        \label{fig:real_tradeoff_youtube}
    \end{subfigure}\hfill
    \begin{subfigure}[t]{0.48\linewidth}
        \centering
        \includegraphics[width=\linewidth]{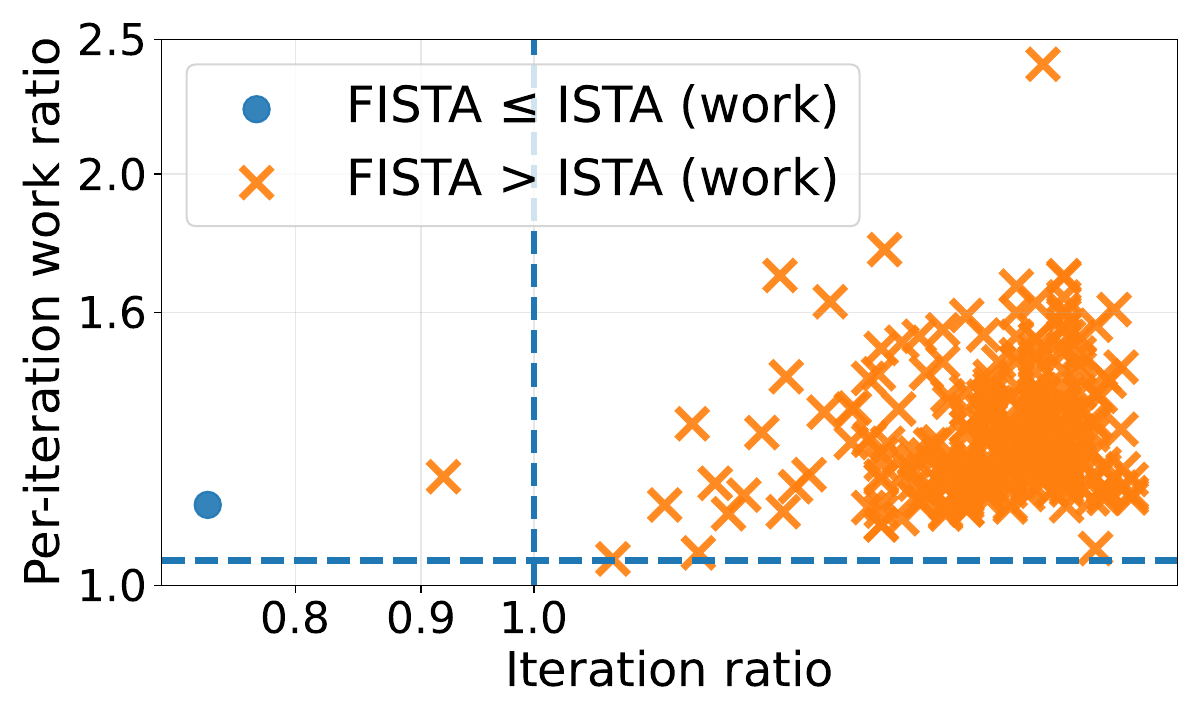}
        \caption{\texttt{com-Orkut}.}
        \label{fig:real_tradeoff_orkut}
    \end{subfigure}

    \caption{\textit{Iterations vs.\ per-iteration work tradeoff.}
    Each point is a seed node (same seeds as in the sweep experiments), at $\alpha=10^{-3}$, $\rho=10^{-4}$, and
    $\varepsilon=10^{-8}$. The $x$-axis is the iteration ratio $N_{\mathrm{F}}/N_{\mathrm{I}}$ and the
    $y$-axis is the per-iteration work ratio $(W_{\mathrm{F}}/N_{\mathrm{F}})/(W_{\mathrm{I}}/N_{\mathrm{I}})$.
    Markers distinguish seeds where FISTA is faster/slower in total work.}
    \label{fig:real_tradeoff_iters_cost}
\end{figure}

\textbf{Degree heterogeneity.}
Because our work metric is degree-weighted, transient activations of even a small number of high-degree nodes can dominate
the locality cost. \Cref{fig:real_degree_ccdf} plots the empirical degree complementary CDF for the four datasets and highlights
the substantially heavier tail of \texttt{com-Orkut} (and, to a lesser extent, \texttt{com-Youtube}), which is consistent with the
larger variability and the small-$\alpha$ slowdowns observed in \Cref{fig:real_work_vs_alpha,fig:real_tradeoff_iters_cost}.

\begin{figure}[h!]
    \centering
    \includegraphics[width=0.55\linewidth]{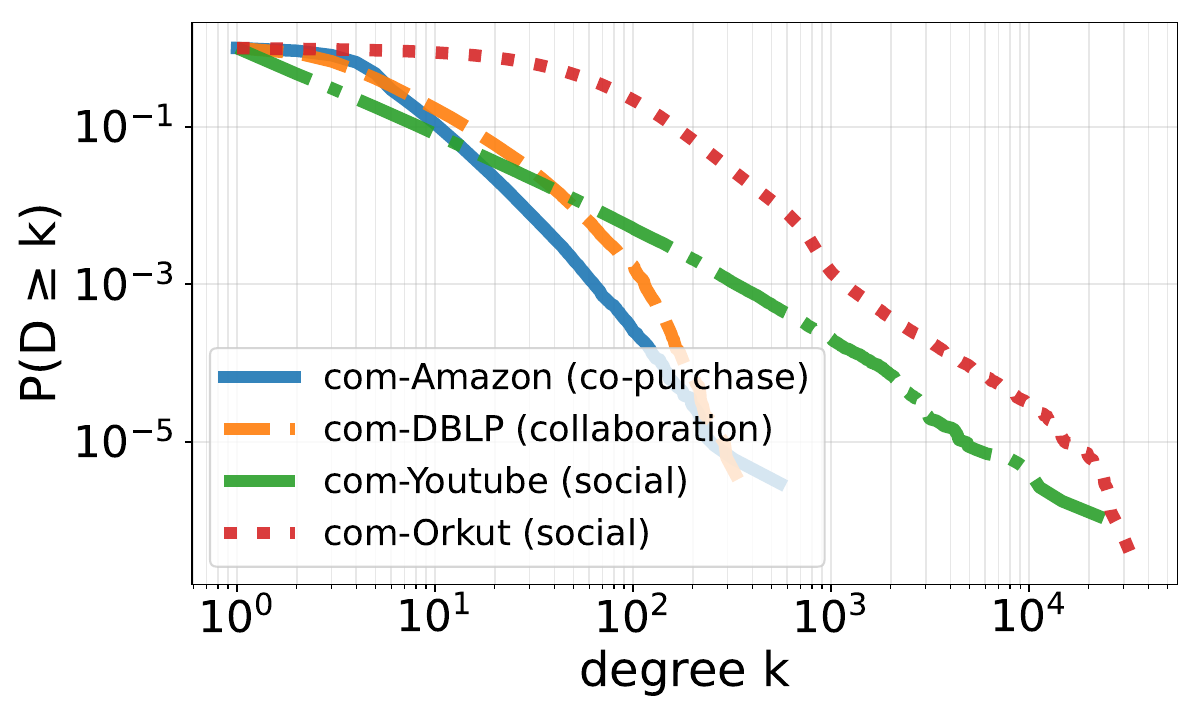}
    \caption{\textit{Degree distributions on the real datasets.}
    The heavier-tailed degree profiles (notably \texttt{com-Orkut}) amplify the impact of transient exploration.}
    \label{fig:real_degree_ccdf}
\end{figure}

\section{AI-assisted development and prompt traceability}
\label{app:ai_prompt_trace}

This paper was developed with the assistance of an interactive large-language-model (LLM) workflow.
The LLM was used as a proof-synthesis and rewriting aid: it generated candidate lemmas, algebraic
manipulations, and \LaTeX\ skeletons, while the human author(s) provided the research direction, imposed
algorithmic constraints, requested specific locality-aware bounds, identified missing assumptions, and
validated (or rejected) intermediate arguments.  The final statements and proofs appearing in the paper
were human-checked and edited for correctness and presentation.

\subsection{Prompt clusters and how they map to results in the paper}
\label{app:prompt_trace_map}

The interactive prompting that led to the final results naturally grouped into a small number of
``prompt clusters.''  Below we summarize each cluster, the key human supervision intervention(s),
and the resulting manuscript artifacts (with cross-references). We use GPT-5.2 Pro (extended thinking) for all results and experiments. 

\paragraph{(P1) ``Standard accelerated algorithm only; avoid expensive subproblems.''}
The initial constraint was to analyze classic one-gradient-per-iteration acceleration (FISTA)
rather than outer-inner schemes or methods that solve expensive auxiliary subproblems.
This constraint fixed the algorithmic object of study and ruled out approaches akin to
expanding-subspace or repeated restricted solves.  It directly shaped the scope of the main runtime result
\Cref{thm:double_reg_work} and the fact that all bounds are expressed in the degree-weighted work model.

\paragraph{(P2) ``Use the margin/KKT slack idea.''}
This idea was suggested by GPT, but we found it useful and, therefore, retained it in the final results.
A key prompt requested a self-contained argument based on a margin parameter.  This produced the
degree-normalized slack definition \Cref{eq:gamma_i} and its operational meaning: an inactive coordinate
can become active at an extrapolated point only if its forward-gradient map deviates from the optimum by
an amount proportional to its slack.  The corresponding quantitative statement is
\Cref{lem:coord_jump}, which is the main bridge from optimality structure to spurious activations.

\paragraph{(P3) ``Transient support is the bottleneck; bound the sum of supports/volumes.''}
A crucial human intervention was to point out that it is not enough to argue eventual identification:
one must control the cumulative degree-weighted work over the entire transient.
This prompted the transition from pointwise identification to a global summation argument:
Cauchy-Schwarz converts ``activation implies a jump'' (from \Cref{lem:coord_jump}) into a bound on
$\sum_k \vol(A_k)$, and geometric contraction of FISTA controls the resulting series.
This is the backbone of the spurious-work bound in the proof of \Cref{thm:double_reg_work}
(see in particular the derivation around \Cref{eq:spurious_work}).

\paragraph{(P4) ``Avoid vacuous bounds when the minimum margin is tiny; use over-regularization.''}
Another human-directed prompt asked how to proceed when the minimum slack can be arbitrarily small,
which would make any bound that depends on $\min_{i\in I^\star}\gamma_i$ meaningless. Thus the idea for analyzing a more regularized problem (``(B)'') but treating nearly-active nodes as part of the target support of the less-regularized problem (``(A)'') was suggested to the LLM. Concretely, this yielded the split in \Cref{lem:two_tier_split}, which uses
regularization-path monotonicity (cf.\ \Cref{lem:path_monotone}) to show that ``small (B)-margin'' nodes
must lie in $S_A$ and should not be charged as spurious.  This is a key input to the work bound
\Cref{thm:double_reg_work}.

\paragraph{(P5) ``Turn the work bound into a running-time bound using $\vol(S^\star)\le 1/\rho$.''}
A prompt explicitly requested that the final complexity be stated in the degree-weighted work model and
use the known sparsity guarantee $\vol(S^\star)\le 1/\rho$.
This guided the decomposition \Cref{eq:running_time} into ``work on the target support'' plus ``spurious
work,'' and it is the reason the first term in \Cref{thm:double_reg_work} scales as
$\widetilde{O}((\rho\sqrt{\alpha})^{-1})$ (up to logarithms).

\paragraph{(P6) ``Give a explicit confinement condition so spurious activations stay local.''}
After the spurious-work summation bound was obtained, a prompt requested a graph-explicit
assumption guaranteeing that all spurious activations remain confined to a boundary set.
This produced the exposure/no-percolation-style sufficient condition formalized as
\Cref{thm:boundary_confinement_exposure}, which is referenced immediately after \Cref{thm:double_reg_work}
to justify the boundary-set hypothesis $\widetilde{A}_k\subseteq\mathcal{B}$.

\paragraph{(P7) ``Identify explicit bad instances where $\gamma$ can be very small (or $0$).''}
To stress-test the margin-based reasoning, a sequence of prompts asked for explicit graphs where the
slack is smaller than $\sqrt{\rho}$ and even $o(\sqrt{\rho})$.  This led to the breakpoint
constructions recorded in \Cref{sec:bad_instances}, including the star graph
(\Cref{lem:star_gamma_small}) and the path graph (\Cref{lem:path_gamma_small}).
These examples motivate why the paper avoids global dependence on $\gamma$ and instead relies on the
over-regularization/two-tier strategy (\Cref{lem:two_tier_split}) together with confinement
(\Cref{thm:boundary_confinement_exposure}).

\paragraph{(P8) ``High-degree non-activation under over-regularization.''}
A later prompt suggested to use the overregularization idea to rule out spurious activations of very high-degree nodes.
This yielded the explicit degree cutoff condition in \Cref{lem:fista_degree_cutoff}, which provides an
additional structural non-activation guarantee that complements the boundary-confinement approach.

\paragraph{(P9) ``Experiments.''} All code was generated by the LLM. However, the authors heavily supervised the process.

\subsection{How much human supervision was required?}
\label{app:prompt_trace_supervision}

The development required human-in-the-loop supervision.
Across roughly two dozen interactive turns, the human prompts performed tasks that the LLM did not do
reliably on its own:
\begin{itemize}[leftmargin=*]
\item \textbf{Problem framing and constraints.}
The human author fixed the algorithmic scope (standard FISTA; no expensive subproblems) and demanded a
locality-aware work bound rather than a standard iteration bound (driving \Cref{thm:double_reg_work}).

\item \textbf{Identifying the real bottleneck.}
A key correction was the insistence that bounding eventual identification is insufficient; one must
bound the sum of transient supports/volumes (leading to the summation argument in the proof of
\Cref{thm:double_reg_work}).

\item \textbf{Stress-testing with counterexamples.}
The human prompts requested explicit worst cases (star and path) and used them to diagnose when naive
$\gamma$-based bounds become vacuous (motivating \Cref{sec:bad_instances} and the over-regularization
strategy used in \Cref{lem:two_tier_split}).

\item \textbf{Assumption checking and proof repair.}
When an intermediate proof relied on an unproven positivity/sign assumption, the human author demanded
either a proof or a repair; this resulted in a revised subgradient/KKT-based certificate (ultimately
not needed for the core theorems, but an important correctness checkpoint).

\item \textbf{\LaTeX\ integration/debugging.}
Compile errors and presentation issues (e.g., list/itemization mistakes) were identified via human
compilation and corrected in subsequent iterations.
\end{itemize}

Overall, the LLM contributed most effectively as a rapid generator of candidate proofs
and algebraic manipulations, while the human supervision was essential for (i) setting the right
target statement, (ii) insisting on the correct work metric, (iii) enforcing locality constraints,
(iv) catching missing assumptions, and (v) selecting which generated material belonged in the final paper.

\section{Formalization of results}\label{sec:formalization}

We formalized the full theorem-level mathematical core of the paper in Lean.
\ifarxiv
The formal versions of the results and their proof can be found here \url{https://github.com/kfoynt/formalized_l1_accelerated}. 
\else
The formal versions of the results and their proof can be found in the supplementary material.
\fi
The development covers the preliminary facts \cref{lemma:initial_gap,fact:fista_rate,lem:yk_distance,lem:path_monotone,lem:prox_grad_monotonicity}, the upper-bound argument \cref{lem:coord_jump,lem:two_tier_split,thm:double_reg_work,thm:boundary_confinement_exposure}, the high-degree non-activation result \cref{lem:fista_degree_cutoff}, the breakpoint constructions \cref{lem:star_gamma_small,lem:path_gamma_small}, and the lower-bound chain \cref{lem:star_leaf_breakpoint,lem:activation_criterion,lem:fista_overshoot,prop:lower_bound}. The experiments and the surrounding expository discussion are not part of the formalization.

The development relies on nine imported statements that are not proved within the project but are either trivial or known from previous work. In the Lean code, these are introduced as axioms in the technical sense of declarations accepted without proof within this development, not as conjectural mathematical assumptions. Concretely, the imported statements are the quadratic expansion of the PageRank quadratic; the strong-convexity gap inequality at a minimizer; the implication ``minimizer implies proximal-gradient fixed point''; the RPPR support-volume bound $\vol(\supp(x^\star))\le 1/\rho$; coordinatewise nonnegativity of the RPPR minimizer; the upper inactive KKT bound $\nabla_i f(x^\star)\le 0$ when $x_i^\star=0$; the strongly-convex FISTA convergence rate; the fact that ISTA iterates stay inside the optimal support; and the linear convergence rate of ISTA. No result specific to the present paper was introduced this way. Once these background facts are imported, the new contributions of the paper, including the complementarity-jump argument, the two-tier split, the work theorem, the confinement theorem, the degree cutoff, the explicit bad instances, and the lower bound, are all proved in Lean.

First, one of the imported statements is purely algebraic: the exact second-order expansion of the quadratic objective. This is a routine identity obtained by expanding a quadratic form, and it could be proved directly from the definitions. Its use as an imported statement is only a bookkeeping choice and it does not hide any substantive mathematical content.

Second, several imported statements are standard facts from first-order convex optimization. The strong-convexity gap inequality is the usual consequence of strong convexity; the proximal-gradient fixed-point characterization is the standard equivalence between optimality and vanishing gradient mapping for convex composite problems, the linear convergence rate of ISTA in the strongly convex case is classical, and the strongly-convex FISTA rate used here is standard textbook material. The original FISTA method is due to \citet{beck2009fast}.

Third, two imported statements are exactly previously published RPPR locality results. We import the support-volume bound $\vol(\supp(x^\star))\le 1/\rho$ and the support containment property for ISTA iterates from the variational analysis of \citet[Theorems~1 and~2]{fountoulakis2019variational}. In our formalization, these facts are used only to translate formally checked iterate-level arguments into degree-weighted work bounds and, in the lower-bound section, to compare FISTA with the known locality guarantee for ISTA. In particular, the FISTA part of the lower bound is formalized directly; only the comparison to ISTA uses these imported RPPR facts.

Finally, the remaining RPPR imported statements, namely minimizer nonnegativity and the upper inactive KKT bound, are also standard properties of the $\ell_1$-regularized PageRank objective for nonnegative seeds. They are explicit in the RPPR literature, see for instance the nonnegativity and KKT lemmas in \citet{ha2021statistical}, and they are consistent with the variational characterization in \citet{fountoulakis2019variational}. These imported statements simply expose the usual sign information at the minimizer in a compact form.

Overall, the formalization should be interpreted as follows. The imported statements collect generic convex-analysis facts and previously established RPPR theorems, while the paper's new acceleration-specific arguments are checked end to end in Lean.

\end{document}

%% file: config.tex
\usepackage{enumitem}
\usepackage{amssymb}
\usepackage{amsmath}

\usepackage{mathrsfs}		
\usepackage{dsfont}

\usepackage{url}					
\hypersetup{colorlinks=true, urlcolor=blue, linktoc = all}
\usepackage{booktabs}

\usepackage{hhline}
\usepackage{multirow}

\usepackage[utf8]{inputenc} 
\usepackage[T1]{fontenc}    

\usepackage{multicol}
\usepackage{silence}
\usepackage{amsfonts,thmtools}
\usepackage{mathtools}
\usepackage{xparse}
\usepackage{enumitem} 
\usepackage{etoolbox}
\usepackage{complexity}
\usepackage{svg}
\usepackage{xcolor, soul}
\usepackage{tablefootnote}
\usepackage[bb=boondox]{mathalpha} 
\definecolor{lightgrey}{rgb}{0.9,0.9,0.9}
\sethlcolor{lightgrey}
\definecolor{mygray}{rgb}{0.6,0.6,0.6}

\crefname{equation}{}{}
\crefname{enumi}{Statement}{Statements} %
\crefname{equation}{}{}
\crefname{enumi}{Statement}{Statements} %

\newtheorem{fact}[theorem]{Fact}

\newcommand{\jmlrBlackBox}{\rule{1.5ex}{1.5ex}}

\newcommand{\jmlrQED}{\hfill\jmlrBlackBox\par\bigskip}
\providecommand{\proofname}{Proof}
\newenvironment{proof}%
{%
 \par\noindent{\bfseries\upshape \proofname\ }%
}%
{\jmlrQED}

\usepackage{tikz}			
\usepackage{times}

\newcommand{\norm}[1]{\| #1 \|} 
\newcommand{\abs}[1]{\lvert #1 \rvert} 
 
\newcommand{\absl}[1]{\left\lvert #1 \right\rvert} 
\newcommand*\circledaux[1]{\tikz[baseline=(char.base)]{
    \node[shape=circle,draw,inner sep=0.8pt] (char) {#1};}}

\NewDocumentCommand{\circled}{m o }{%
    \IfNoValueTF{#2}{\circledaux{#1}}{\stackrel{\circledaux{#1}}{#2}}%
}

\usepackage{cancel}
\newcommand\Ccancel[2][black]{
    \let\OldcancelColor\CancelColor
    \renewcommand\CancelColor{\color{#1}}
    \cancel{#2}
    \renewcommand\CancelColor{\OldcancelColor}
}

\newcommand{\defi}{:=}

\renewcommand*\R{\mathbb{R}}                              %
\let\epsilon\varepsilon

\newcommand{\ones}{\mathds{1}}                           %

\usepackage{pifont} 
\DeclareMathOperator*{\argmin}{arg\,min}                %

\makeatletter
\renewcommand\paragraph{\@startsection{paragraph}{4}{\z@}%
                                    {0ex \@plus0.5ex \@minus.2ex}%
                                    {-1em}%
                                    {\normalfont\normalsize\bfseries}}
\makeatother

\definecolor{labelkey}{rgb}{0,0.08,0.45}
\definecolor{refkey}{rgb}{0,0.6,0.0}
\definecolor{Brown}{rgb}{0.45,0.0,0.05}
\definecolor{dgreen}{rgb}{0.00,0.49,0.00}
\definecolor{dblue}{rgb}{0,0.08,0.75}
\definecolor{ffwwqq}{rgb}{1.,0.4,0.}
\definecolor{qqzzqq}{rgb}{0.,0.6,0.}
\definecolor{qqqqff}{rgb}{0.,0.,1.}
\definecolor{dred}{HTML}{D90404}
\definecolor{orng}{HTML}{D35400}
\definecolor{cb-black}      {RGB}{  0,   0,   0}
\definecolor{cb-blue-green} {RGB}{  0,  073,  073}
\definecolor{cb-green-sea}  {RGB}{  0, 146, 146}
\definecolor{cb-rose}       {RGB}{255, 109, 182}
\definecolor{cb-salmon-pink}{RGB}{255, 182, 119}
\definecolor{cb-purple}     {RGB}{ 73,   0, 146}
\definecolor{cb-blue}       {RGB}{ 0, 109, 219}
\definecolor{cb-lilac}      {RGB}{182, 109, 255}
\definecolor{cb-blue-sky}   {RGB}{109, 182, 255}
\definecolor{cb-blue-light} {RGB}{182, 219, 255}
\definecolor{cb-burgundy}   {RGB}{146,   0,   0}
\definecolor{cb-brown}      {RGB}{146,  73,   0}
\definecolor{cb-clay}       {RGB}{219, 209,   0}
\definecolor{cb-green-lime} {RGB}{ 36, 255,  36}
\definecolor{cb-yellow}     {RGB}{255, 255, 109}
\definecolor{bred}{HTML}{FF0000}
\definecolor{bpurp}{HTML}{BF00BF}
\definecolor{bblu}{HTML}{0000FF}
\definecolor{bcyan}{HTML}{00BFBF}
\definecolor{byellow}{HTML}{BFBF00}
\definecolor{bgreen}{HTML}{008000}

\newcommand{\inlinesmash}[1]{\smash{$#1$}}

%% file: algorithm_config.tex
\usepackage{algorithm}
\usepackage{algcompatible}
\algnewcommand{\lst}{\texttt{lst}}
\algnewcommand{\slst}{\texttt{slst}}
\algnewcommand{\SEND}{\textbf{send}}

\newsavebox{\algleft}
\newsavebox{\algright}

\makeatletter
\newcounter{algorithmicH}%
\let\oldalgorithmic\algorithmic
\renewcommand{\algorithmic}{%
  \stepcounter{algorithmicH}%
  \oldalgorithmic}%
\renewcommand{\theHALG@line}{ALG@line.\thealgorithmicH.\arabic{ALG@line}}
\makeatother

%% file: bibliography_config.tex
\usepackage[natbib, backend=biber, maxcitenames=3, minalphanames=3, maxbibnames=99, style=alphabetic, hyperref, backref, useprefix=true, uniquename=false, doi=false,url=false,eprint=false]{biblatex} 

\usepackage{csquotes}               
\bibliography{refs}        

\DeclareCiteCommand{\cite}
  {\usebibmacro{prenote}}
  {\usebibmacro{citeindex}%
   \printtext[bibhyperref]{\usebibmacro{cite}}}
  {\multicitedelim}
  {\usebibmacro{postnote}}

\DeclareCiteCommand*{\cite}
  {\usebibmacro{prenote}}
  {\usebibmacro{citeindex}%
   \printtext[bibhyperref]{\usebibmacro{citeyear}}}
  {\multicitedelim}
  {\usebibmacro{postnote}}

\DeclareCiteCommand{\parencite}[\mkbibparens]
  {\usebibmacro{prenote}}
  {\usebibmacro{citeindex}%
    \printtext[bibhyperref]{\usebibmacro{cite}}}
  {\multicitedelim}
  {\usebibmacro{postnote}}

\DeclareCiteCommand*{\parencite}[\mkbibparens]
  {\usebibmacro{prenote}}
  {\usebibmacro{citeindex}%
    \printtext[bibhyperref]{\usebibmacro{citeyear}}}
  {\multicitedelim}
  {\usebibmacro{postnote}}

\DeclareCiteCommand{\citeauthor}
  {\usebibmacro{prenote}}
  {\ifciteindex
     {\indexnames{labelname}}
     {}%
   \printtext[bibhyperref]{\printnames{labelname}}}
  {\multicitedelim}
  {\usebibmacro{postnote}}

\DeclareCiteCommand{\footcite}[\mkbibfootnote]
  {\usebibmacro{prenote}}
  {\usebibmacro{citeindex}%
  \printtext[bibhyperref]{ \usebibmacro{cite}}}
  {\multicitedelim}
  {\usebibmacro{postnote}}

\DeclareCiteCommand{\footcitetext}[\mkbibfootnotetext]
  {\usebibmacro{prenote}}
  {\usebibmacro{citeindex}%
   \printtext[bibhyperref]{\usebibmacro{cite}}}
  {\multicitedelim}
  {\usebibmacro{postnote}}

\DeclareCiteCommand{\textcite}
  {\boolfalse{cbx:parens}}
  {\usebibmacro{citeindex}%
   \printtext[bibhyperref]{\usebibmacro{textcite}}}
  {\ifbool{cbx:parens}
     {\bibcloseparen\global\boolfalse{cbx:parens}}
     {}%
   \multicitedelim}
  {\usebibmacro{textcite:postnote}}

\newbibmacro{string+doiurlisbn}[1]{%
  \iffieldundef{doi}{%
    \iffieldundef{url}{%
      \iffieldundef{isbn}{%
        \iffieldundef{issn}{%
          #1%
        }{%
          \href{http://books.google.com/books?vid=ISSN\thefield{issn}}{#1}%
        }%
      }{%
        \href{http://books.google.com/books?vid=ISBN\thefield{isbn}}{#1}%
      }%
    }{%
      \href{\thefield{url}}{#1}%
    }%
  }{%
    \href{https://doi.org/\thefield{doi}}{#1}%
  }%
}

\DeclareFieldFormat{title}{\usebibmacro{string+doiurlisbn}{\mkbibemph{#1}}}
\DeclareFieldFormat[article,incollection,inproceedings]{title}%
    {\usebibmacro{string+doiurlisbn}{#1}}